\documentclass[12pt]{article}
\usepackage{amsmath,amssymb,amsthm}
\numberwithin{equation}{section}

\newcommand{\g}{\mathfrak{g}}
\newcommand{\pe}{\mathfrak{p}}
\newcommand{\mr}{\mathfrak{r}}
\newcommand{\bq}{\mathfrak{b}(q)}

\newcommand{\mN}{\mathfrak{N}}
\newcommand{\mo}{\mathfrak{o}}

\newcommand{\p}{\varphi}
\newcommand{\gx}{\g^{\times}}
\newcommand{\lw}{L \cap W}
\newcommand{\e}{\varepsilon}
\newcommand{\dl}{\delta}
\newcommand{\ka}{\kappa}
\newcommand{\la}{\lambda}
\newcommand{\om}{\omega}
\newcommand{\mm}{\mu_{2}}
\newcommand{\mmm}{\mu_{3}}

\newcommand{\aw}{A^{+}(W)}
\newcommand{\awc}{A^{+}(W)^{\circ}}

\newcommand{\am}{A^{+}(M)}
\newcommand{\amd}{A^{+}(M)\widetilde{\,}}
\newcommand{\alw}{A^{+}(L \cap W)}
\newcommand{\alwd}{A^{+}(L \cap W)\widetilde{\,}}
\newcommand{\gf}{\g_{4}^{1}}
\newcommand{\gt}{\g_{3}^{1}}

\newcommand{\gb}{g_{2}}

\newcommand{\gbp}{g_{3}^{\prime}}
\newcommand{\ib}{1_{B}}

\newcommand{\n}{\nu}
\newcommand{\nc}{\nu^{\circ}}
\newcommand{\sq}{\sqrt{-q}}
\newcommand{\sqd}{\sqrt{\delta}}

\newcommand{\tp}[1]{{\rlap{$\phantom{#1}$}}^t #1}
\newcommand{\REF}[1]{\eqref{#1}}
\newcommand{\THM}[1]{Theorem \ref{#1}}

\newcommand{\LEM}[1]{Lemma \ref{#1}}

\title{
On the genera of certain integral lattices in ternary quadratic spaces 
} 
\author{Manabu Murata} 
\date{}
\theoremstyle{definition}
\newtheorem{sect}{}[section]

\theoremstyle{plain}
\newtheorem{thm}[sect]{Theorem}
\newtheorem{lem}[sect]{Lemma}

\newtheorem{cor}[sect]{Corollary}

\newtheorem*{theorem*}{Theorem}

\begin{document}
\maketitle
\begin{abstract}
This paper treats certain integral lattices with respect to ternary quadratic forms, 
which are obtained from the data of a non-zero element and 
a maximal lattice in a quaternary quadratic space. 
Such a lattice can be described by means of an order associated with 
the lattice in the even Clifford algebra of the ternary form. 
This provides a correspondence between the genus of the lattice and that of the order. 
\end{abstract}


\section*{Introduction}

Let $(V,\, \p)$ be a quadratic space 
consisting of a vector space $V$ of dimension $4$ over an algebraic number field $F$ and 
a nondegenerate symmetric $F$-bilinear form $\p$ on $V$, 
and $L$ a maximal lattice in $V$ with respect to $\p$. 
Put $\p[x] = \p(x,\, x)$ for $x \in V$. 
For an element $h$ of $V$ such that $\p[h] \ne 0$, 
we put 
\begin{gather*}
W = (Fh)^{\perp} = \{x \in V \mid \p(x,\, h) = 0\} 
\end{gather*}
and consider a quadratic space $(W,\, \psi)$ of dimension $3$ over $F$ 
with the restriction $\psi$ of $\p$ to $W$. 
Then we have an integral $\g$-lattice $L \cap W$ in $W$ with respect to $\psi$. 
Here $\g$ is the ring of all integers in $F$ and 
the terms \textit{integral} and \textit{maximal} are given as follows: 
Given $\g$-lattice $L$ in $V$, 
we call it \textit{integral} with respect to $\p$ if $\p[L] \subset \g$; 
notice that $L \subset \widetilde{L}$ if $L$ is integral, 
where $\widetilde{L} = L\,\widetilde{\,} = \{ x \in V \mid 2\p(x,\, L) \subset \g \}$; 
by a \textit{$\g$-maximal}, or simply \textit{maximal}, lattice $L$ with respect to $\p$, 
we understand a $\g$-lattice $L$ in $V$ 
which is maximal among integral lattices with respect to $\p$. 
To $\lw$ above we can associate an order $\alw$ in the even Clifford algebra $A^{+}(W)$ of $\psi$ 
as will be defined below. 
The purpose of this paper is to describe the lattice $\lw$ by means of the order $\alw$. 

To explain more precisely, 
let us recall some basic facts on ternary quadratic spaces (for details see {\S}1.2); 
$A^{+}(W)$ is a quaternion algebra over $F$, 
its canonical involution $*$ gives the main involution of $A^{+}(W)$ as a quaternion algebra, 
and the even Clifford group $G^{+}(W)$ of $\psi$ consists of all invertible elements of $A^{+}(W)$. 
Moreover $(W,\, \psi)$ is isomorphic to $(A^{+}(W)^{\circ},\, d\nu^{\circ})$ with some $d \in F^{\times}$, 
where $\nu$ is the norm form of $A^{+}(W)$, 
$A^{+}(W)^{\circ} =\{x \in A^{+}(W) \mid x^{*} = -x\}$, 
and $\nu^{\circ}$ is the restriction of $\nu$ to $A^{+}(W)^{\circ}$. 
Under such an isomorphism, 
the special orthogonal group 
$SO^{\psi}=\{ \gamma \in O^{\psi} \mid \det(\gamma) = 1 \}$ 
of $\psi$ can be identified with 
the set of all the mappings $x \mapsto \alpha^{-1}x\alpha$ for $\alpha \in A^{+}(W)^{\times}$, 
where $O^{\psi}=\{ \gamma \in GL(W) \mid \psi(x\gamma,\, y\gamma) = \psi(x,\, y) 
	\text{ for every } x,\, y\in W\}$. 
Given an integral lattice $N$ in $(W,\, \psi)$, 
let $A(N)$ be the order in the Clifford algebra $A(W)$ of $\psi$ generated by $\g$ and $N$, 
and set $A^{+}(N) = \aw \cap A(N)$. 
Then $A^{+}(N)$ is an order in $\aw$. 

Now, 
fix an isomorphism $\xi$ of $(W,\ \psi)$ onto 
$(A^{+}(W)^{\circ},\ d\nu^{\circ})$ with $d \in F^{\times}$ 
and put $d^{-1}\g = \mathfrak{a}\mathfrak{r}^{2}$ with 
a squarefree integral ideal $\mathfrak{a}$ and a $\g$-ideal $\mathfrak{r}$ of $F$. 
Then we shall show in \THM{th1} that 
\begin{gather}
(L \cap W)\xi = \mathfrak{r}\mathfrak{c}D_{\psi}[M/L \cap W]\cdot
	\{A^{+}(L \cap W)\widetilde{\,} \cap A^{+}(W)^{\circ}\}. \label{rec0}
\end{gather}
Here $\mathfrak{c}$ is the product of the prime ideals of $\mathfrak{a}$ 
which are prime to the discriminant $D_{\psi}$ of $A^{+}(W)$, 
$[M/L \cap W]$ is the integral ideal in $F$ given in \REF{42} 
with a maximal lattice $M$ in $(W,\, \psi)$, 
and $A^{+}(L \cap W)\widetilde{\,} = \{x \in A^{+}(W) \mid 2\nu(x,\, \alw) \subset \g\}$. 

From the description of $\lw$ in \REF{rec0} we have the following results: 
\begin{enumerate}
\item The $SO^{\psi}$-genus of $L \cap W$ is determined by the genus of $\mathfrak{o}$ 
	which is defined by the set 
	$\{\alpha^{-1}\mathfrak{o}\alpha \mid \alpha \in A^{+}(W)^{\times}_{\mathbf{A}}\}$. 
	Here $\mathfrak{o} = \alw$ and the subscript $\mathbf{A}$ means adelization. 
\item $C(L \cap W) = \tau(T(\mathfrak{o}))$. 
	Here $C(L \cap W) = \{\gamma \in SO^{\psi}_{\mathbf{A}} \mid (L \cap W)\gamma = L \cap W\}$, 
	$T(\mathfrak{o}) = \{\alpha \in G^{+}(W)_{\mathbf{A}} \mid \alpha\mathfrak{o} = \mathfrak{o}\alpha\}$, 
	and $\tau$ is a surjective homomorphism of $G^{+}(W)_{\mathbf{A}}$ onto $SO^{\psi}_{\mathbf{A}}$ 
	defined in \REF{tau}. 
\item The map $N \mapsto A^{+}(N)$ gives a bijection of the set of $SO^{\psi}$-classes 
	in the $SO^{\psi}$-genus of $\lw$ onto the set of isomorphism classes in the genus of $\mathfrak{o}$. 
\end{enumerate}
In addition it is noted that $\mathfrak{o}$ has discriminant $\p[h][\widetilde{L}/L](2\p(h,\, L))^{-2}$, 
which is not squarefree in general. 
Here $[\widetilde{L}/L]$ is the discriminant ideal of $\p$ defined in Section 1.1 of the text. 

All results in the present paper concern the integral lattices $\lw$, 
which are obtained from various $(V,\, \p)$, $L$, and $h$. 
Though our investigation does not treat \textit{all} integral lattices in $(W,\, \psi)$, 
such lattices and the associated orders can be applied to the study of 
quadratic Diophantine equations in four variables, 
as was shown in \cite{pr} and \cite{qeo}. 
In fact, 
our results complement \cite[Lemma 2.1]{pr}, \cite[Lemma 2.1,\ and\ Theorem 2.2]{qeo} 
that played fundamental roles in the proofs of the main results in those papers. 
\\ 

\textit{Notation}. 
We denote by $\mathbf{Z}$ and $\mathbf{Q}$ the ring of rational integers and the field of rational numbers, 
respectively. 
If $R$ is an associative ring with identity element and if $M$ is an $R$-module, 
then we write $R^{\times}$ for the group of all invertible elements of $R$ and 
$M_{n}^{m}$ the $R$-module of $m \times n$-matrices with entries in $M$. 
We set $R^{\times 2} = \{a^{2} \mid a \in R^{\times}\}$. 
For a finite set $X$, 
we denote by $\# X$ the number of elements in $X$. 
We also write $\mathrm{diag}[a_{1},\, \cdots ,\, a_{s}]$ 
for the matrix with square matrices $a_{1},\, \cdots ,\, a_{s}$ in the diagonal blocks 
and $0$ in all other blocks. 

Let $V$ be a vector space over a field $F$ of characteristic $0$, 
and $GL(V)$ the group of all $F$-linear automorphisms of $V$. 
We let $GL(V)$ act on $V$ on the \textit{right}. 

Let $F$ be an algebraic number field (of finite degree) and $\g$ the ring of all algebraic integers in $F$. 
For a fractional ideal in $F$ we often call it a $\g$-ideal. 
Let $\mathbf{a}$, $\mathbf{h}$, and $\mathbf{r}$ 
be the sets of archimedean primes, nonarchimedean primes, and real archimedean primes of $F$, 
respectively. 
We denote by $F_{v}$ the completion of $F$ at $v \in \mathbf{a} \cup \mathbf{h}$. 
We often identify $v$ with the prime ideal of $F$ corresponding to $v \in \mathbf{h}$. 
For $v \in \mathbf{h}$, 
we denote by $\g_{v}$, $\mathfrak{p}_{v}$, and $\pi_{v}$ 
the maximal order of $F_{v}$, the prime ideal in $F_{v}$, and a prime element of $F_{v}$, 
respectively. 
If $K$ is a quadratic extension of $F$, 
we denote by $D_{K/F}$ the relative discriminant of $K$ over $F$. 

By a \textit{$\g$-lattice} $L$ in a vector space $V$ over a number field or nonarchimedean local field $F$, 
we mean a finitely generated $\g$-submodule in $V$ containing a basis of $V$. 
By an \textit{order} in a quaternion algebra $B$ over $F$ 
we mean a subring of $B$ containing $\g$ that is a $\g$-lattice in $B$. 
For a symmetric $F$-bilinear form $\p$ on $V$ and two subspaces $X$ and $Y$ of $V$, 
we denote by $X \oplus Y$ the direct sum of $X$ and $Y$ 
if $\p(x,\, y) = 0$ for every $x \in X$ and $y \in Y$; 
we also denote by $\p|_{X}$ the restriction of $\p$ to $X$. 
When $X$ is an object defined over a number field $F$, 
we often denote by $X_{v}$ the localization at a prime $v$ if it is meaningful. 

\section{Quaternary and ternary quadratic spaces}

\subsection{Preliminaries for quadratic spaces}
Throughout the paper we shall use the same notation and terminology as in \cite{qeo}, 
which are following the textbook by Shimura \cite{04}. 
We begin by introducing the invariants and the discriminant ideal of a quadratic space $(V,\, \p)$ 
over an algebraic number field or its completion $F$. 
We often call the former $F$ a global field and 
the latter a local field when it is nonarchimedean. 

By the \textit{invariants} of $(V,\, \p)$ over a global field $F$, 
we understand a set of data 
\begin{gather}
\{n,\ F(\sqrt{\delta}),\ Q(\p),\ \{s_{v}(\p)\}_{v \in \mathbf{r}}\}, \label{inv}
\end{gather}
where $n$ is the dimension of $V$, 
$F(\sqrt{\delta})$ is the \textit{discriminant field} of $\p$ with $\delta = (-1)^{n(n-1)/2}\det(\p)$, 
$Q(\p)$ is the \textit{characteristic quaternion algebra} of $\p$, 
and $s_{v}(\p)$ is the \textit{index} of $\p$ at $v \in \mathbf{r}$. 
For these definitions, 
the reader is referred to \cite[{\S}1.1,\ {\S}3.1,\ and\ {\S}4.1]{cq}. 
By virtue of \cite[Theorem 4.2]{cq}, 
the isomorphism class of $(V,\, \p)$ is determined by 
$\{n,\, F(\sqrt{\delta}),\, Q(\p),\, \{s_{v}(\p)\}_{v \in \mathbf{r}}\}$ 
and vice versa. 

The characteristic algebra $Q(\p_{v})$ is also defined for 
$\p_{v}$ at $v \in \mathbf{a} \cup \mathbf{h}$ (cf.\ \cite[{\S}3.1]{cq}). 
Here we put $V_{v} = V\otimes_{F}F_{v}$ and 
$\p_{v}$ denotes the $F_{v}$-bilinear extension of $\p$ to $V_{v}$. 
If $v \in \mathbf{h}$, by \cite[Lemma 3.3]{cq}, 
the isomorphism class of $(V_{v},\, \p_{v})$ is determined by 
$\{n,\, F_{v}(\sqrt{\delta}),\, Q(\p_{v})\}$. 
As for $v \in \mathbf{a}$, 
it is determined by $\{n,\, s_{v}(\p)\}$ if $v \in \mathbf{r}$, 
and by the dimension $n$ if $v \not\in \mathbf{r}$. 
If $v \in \mathbf{r}$, 
then $Q(\p_{v})$ is given by \cite[(4.2a)\ and\ (4.2b)]{cq}, for example. 
If $v \not\in \mathbf{r}$, 
then $Q(\p_{v}) = M_{2}(\mathbf{C})$, 
where $\mathbf{C}$ is the field of complex numbers. 
\\

Let $SO^{\p}(V)_{\mathbf{A}}$ be the adelization of 
$SO^{\p}(V)$ in the usual sense (cf.\ \cite[{\S}9.6]{04}). 
For $\alpha \in SO^{\p}(V)_{\mathbf{A}}$ and 
a $\g$-lattice $L$ in $V$, 
we denote by $L\alpha$ the $\g$-lattice in $V$ whose localization 
is given by $L_{v}\alpha_{v}$ for every $v \in \mathbf{h}$. 
We put 
\begin{gather}
C(L) = \{ \alpha \in SO^{\p}(V)_{\mathbf{A}} \mid L\alpha = L \}, \quad 
C(L_{v}) = SO^{\p_{v}}(V_{v}) \cap C(L),\quad (v \in \mathbf{h}) \nonumber \\ 
\Gamma(L) = SO^{\p}(V) \cap C(L). \label{stbl} 
\end{gather}
Then the map $\alpha \mapsto L\alpha^{-1}$ gives a bijection of 
$SO^{\p}\setminus SO^{\p}_{\mathbf{A}}/C(L)$ onto 
$\{L\alpha \mid \alpha \in SO^{\p}_{\mathbf{A}}\}/SO^{\p}$. 
We call $\{L\alpha \mid \alpha \in SO^{\p}_{\mathbf{A}}\}$ the $SO^{\p}$-\textit{genus} of $L$, 
$\{L\gamma \mid \gamma \in SO^{\p}\}$ the $SO^{\p}$-\textit{class} of $L$, 
and $\#\{SO^{\p}\setminus SO^{\p}_{\mathbf{A}}/C(L)\}$ 
the \textit{class number} of $SO^{\p}$ relative to $C(L)$ 
or the class number of the genus of $L$ with respect to $SO^{\p}$. 
It is known that 
all $\g$-maximal lattices in $V$ with respect to $\p$ form a single $SO^{\p}$-genus. 
Let $A^{+}(V)^{\times}_{\mathbf{A}}$ (resp.\ $G^{+}(V)_{\mathbf{A}}$) 
be the adelization of $A^{+}(V)^{\times}$ (resp.\ $G^{+}(V)$). 
We define a homomorphism $\tau$ of $G^{+}(V)$ onto $SO^{\p}$ by 
\begin{gather}
x\tau(\alpha) = \alpha^{-1}x\alpha \qquad \text{for $x \in V$ and $\alpha\in G^{+}(V)$}. \label{tau}
\end{gather}
This can be extended to a homomorphism of $G^{+}(V)_{\mathbf{A}}$ onto $SO^{\p}_{\mathbf{A}}$ 
(see \cite[{\S}9.10]{04} for details). 
We denote it by the same symbol $\tau$. 

For two $\g$-lattices $L$ and $M$ in $V$ over a global or local field $F$, 
we denote by $[L/M]$ a $\g$-ideal of $F$ generated over $\g$ by $\det(\alpha)$ 
of all $F$-linear automorphisms $\alpha$ of $V$ such that $L\alpha \subset M$. 
If $F$ is a global field, 
then $[L/M] = \prod_{v \in \mathbf{h}} [L_{v}/M_{v}]$ 
with the localization $[L/M]_{v} = [L_{v}/M_{v}]$ at each $v$. 
Following \cite[{\S}6.1]{cq}, 
in both global and local $F$, 
we call $[\widetilde{L}/L]$ the \textit{discriminant ideal} of $(V, \p)$ 
if $L$ is a $\g$-maximal lattice in $V$ with respect to $\p$. 
This is independent of the choice of $L$. 
If $F$ is a local field, 
the discriminant ideal of $\p$ coincides with that of a core subspace of $\p$. 
\\

For a $\g$-lattice $L$ in $V$, 
$q \in F$, 
and a $\g$-ideal $\mathfrak{b}$ of $F$, 
we put 
\begin{gather*}
L[q] = \{x \in L \mid \p[x] = q\}, \quad 
L[q,\ \mathfrak{b}] = \{ x \in V \mid \p[x] = q,\ \p(x,\, L) = \mathfrak{b} \}. 
\end{gather*}
Here $\p(x,\, L) = \{\p(x,\, y) \mid y \in L\}$, 
which becomes a $\g$-ideal of $F$. 
Suppose $F$ is a nonarchimedean local field. 
Let $V$ have dimension $n > 2$ and $L$ be a $\g$-maximal lattice in $V$ with respect to $\p$. 
Then \cite[Theorem\ 10.5]{04} due to Shimura shows that 
\begin{gather}
L[q,\ \mathfrak{b}] = hC(L), \label{th13}
\end{gather}
provided $h \in L[q,\, \mathfrak{b}]$ (cf.\ also\ \cite[Theorem 1.3]{6b}). 
In the same setting, 
Yoshinaga \cite[Theorem 3.5]{y} gives an element of $L[q,\, \mathfrak{b}]$ 
in terms of a Witt decomposition of $(V,\, \p)$ for every set $L[q,\, \mathfrak{b}]$ contained in $L$. 
In later arguments we will employ such elements by means of \REF{th13}. 
\\

Let us introduce some symbols which will be used throughout the paper. 
Let $F$ be a global or local field. 
For a quadratic extension field $K$ of $F$, 
we put $2\kappa(x,\, y) = xy^{\rho} + x^{\rho}y$ for $x,\, y \in K$ 
with a nontrivial automorphism $\rho$ of $K$ over $F$. 
Put also $\kappa[x] = \kappa(x,\, x)$, 
which is the norm $N_{K/F}(x)$ of $x$. 
For a quaternion algebra $B$ over $F$, 
we put $2\beta(x,\, y) = xy^{\iota} + yx^{\iota}$ for $x,\, y \in B$ 
with the main involution $\iota$ of $B$. 
Then the reduced norm $N_{B/F}(x)$ of $x$ is given by $\beta[x] = \beta(x,\, x)$; 
the reduced trace is $\mathrm{Tr}_{B/F}(x)=2\beta(x,\, 1)$. 
We denote by $D_{B}$ the discriminant of $B$. 
Every quaternion algebra $B$ over $F$ can be given by 
\begin{gather}
K + K\omega, \quad \omega^{2} = c, \quad x\omega = \omega x^{\iota} \quad \text{for every $x \in K$} 
\label{om}
\end{gather}
with a quadratic extension field $K$ of $F$, 
$\om \in B^{\times}$, 
and $c \in F^{\times}$ (cf.\ \cite[{\S}1.10]{04}). 
We denote it by $\{K,\, c\}$. 
Then $\{K,\, c\} = M_{2}(F)$ if and only if $c \in \kappa[K^{\times}]$. 
In particular when $F$ is local, 
a division quaternion algebra over $F$ is isomorphic to $\{K,\, c\}$ 
with an arbitrarily fixed element $c \in F^{\times}$ such that $c \not\in \kappa[K^{\times}]$. 
This is because there is a unique division quaternion algebra over $F$ up to isomorphisms; 
see \cite[Theorem 5.14]{04}, for example. 
We set 
\begin{gather}
B^{\circ} = \{x \in B \mid x^{\iota} = -x\},\quad 
	\beta^{\circ} = \beta |_{B^{\circ}}. \label{bc}
\end{gather}

For an order $\mathfrak{o}$ in $B$ it is known that 
$[\widetilde{\mathfrak{o}}/\mathfrak{o}] = d(\mathfrak{o})^{2}$ 
with an integral ideal $d(\mathfrak{o})$ in $F$. 
Here $\widetilde{\mathfrak{o}}=\{ x \in B \mid \mathrm{Tr}_{B/F}(x\mathfrak{o}) \subset \g \}$. 
The ideal $d(\mathfrak{o})$ is called the \textit{discriminant} of $\mathfrak{o}$. 
If $F$ is a number field and $\mathfrak{o}$ is a maximal order, 
then $d(\mathfrak{o})$ is the product of all prime ideals ramified in $B$, 
that is, 
the discriminant of $B$. 
When $F$ is global, 
we set 
\begin{gather}
T(\mathfrak{o}) = \{\alpha \in B^{\times}_{\mathbf{A}} \mid 
	\alpha\mathfrak{o} = \mathfrak{o}\alpha\}, \quad 
T(\mathfrak{o}_{v}) = B^{\times}_{v} \cap T(\mathfrak{o}) 
\quad (v \in \mathbf{a}\cup\mathbf{h}), \nonumber \\ 
\Gamma^{*}(\mathfrak{o}) = B^{\times} \cap T(\mathfrak{o}), \label{nml} 
\end{gather}
where $B^{\times}_{\mathbf{A}}$ is the adelization of $B^{\times}$, 
$B_{v} = B\otimes_{F}F_{v}$, 
and $\mo_{v} = \mo\otimes_{\g}\g_{v}$; 
note that $T(\mathfrak{o}_{v}) = B^{\times}_{v}$ if $v \in \mathbf{a}$. 
Then $\#\{T(\mo)\setminus B^{\times}_{\mathbf{A}}/B^{\times}\}$ 
is called the \textit{type number} of $\mo$. 
If $U = B^{\times}_{\mathbf{a}}\prod_{v \in \mathbf{h}}\mo_{v}^{\times}$ in $B^{\times}_{\mathbf{A}}$, 
then $\#\{U\setminus B^{\times}_{\mathbf{A}}/B^{\times}\}$ is called the \textit{class number} of $\mo$. 

\subsection{Ternary quadratic spaces}

We recall some basic facts on three-dimensional quadratic spaces $(W,\, \psi)$ 
over a number field or its completion $F$. 
The characteristic algebra $Q(\psi)$ is given by $A^{+}(W)$ by definition. 
The core dimension $s_{v}$ of $(W,\, \psi)$ at $v \in \mathbf{h}$ is determined by 
\begin{gather}
s_{v} = 
\begin{cases}
1 & \text{if $Q(\psi_{v}) = M_{2}(F_{v})$}, \\
3 & \text{if $Q(\psi_{v})$ is a division algebra}. 
\end{cases} \label{cd3}
\end{gather}
This can be seen from \cite[{\S}3.2]{cq} and the proof of \cite[Lemma 3.3]{cq}. 

There are isomorphisms of $(W,\, \psi)$ onto $(A^{+}(W)^{\circ},\, d\nu^{\circ})$ with $d \in F^{\times}$. 
Let us explain it, 
following \cite[{\S}7.3]{04}. 
Take an orthogonal basis $\{k_{1},\, k_{2},\, k_{3}\}$ of $W$ with respect to $\psi$; 
namely, 
an $F$-basis $\{k_{i}\}_{i=1}^{3}$ of $W$ such that $\psi(k_{i},\, k_{j}) = 0$ for $i \ne j$. 
Under the identification of $W$ with the corresponding subspace in the Clifford algebra $A(W)$, 
put $\xi = k_{1}k_{2}k_{3} \in A(W)^{\times}$; 
then $F + F\xi$ is the center of $A(W)$. 
We see that $A^{+}(W) = F + Fk_{1}k_{2} + Fk_{1}k_{3} + Fk_{2}k_{3}$ 
and $W\xi = Fk_{1}k_{2} + Fk_{1}k_{3} + Fk_{2}k_{3}$. 
By \cite[Theorem 2.8(ii)]{04}, 
$A^{+}(W)$ is a quaternion algebra over $F$; 
the main involution coincides with the canonical involution $*$ restricted to $A^{+}(W)$. 
Then the map $x \mapsto x\xi$ gives an $F$-linear isomorphism of $W$ onto $A^{+}(W)^{\circ}$ 
such that $(x\xi)(x\xi)^{*} = \xi \xi^{*}\psi[x]$ for $x \in W$. 
Putting $\nu[y] = yy^{*}$ for $y \in A^{+}(W)$, 
we have an isomorphism 
\begin{gather}
(W,\, \psi) \cong (A^{+}(W)^{\circ},\, (\xi \xi^{*})^{-1}\nu^{\circ}) \quad 
\text{via } x \longmapsto x\xi. \label{is3}
\end{gather}
Notice that $\xi \xi^{*} = \psi[k_{1}]\psi[k_{2}]\psi[k_{3}] \in F^{\times}$ 
and so $F(\sqrt{-\xi \xi^{*}})$ is the discriminant field of $\psi$. 
For the sake of simplicity, 
we often denote by $\xi$ such an isomorphism in \REF{is3}. 
Let $\{g_{i}\}$ be another orthogonal basis of $(W,\, \psi)$ and put $\zeta = g_{1}g_{2}g_{3}$. 
Since $F\xi = F\zeta$ by \cite[Lemma 2.5\ (iii)]{04}, 
the isomorphism defined in \REF{is3} with $\zeta$ is given by 
$x \mapsto cx\xi$ with some $c \in F^{\times}$. 

Let $G^{+}(W)$ be the even Clifford group of $(W,\, \psi)$ 
and $\tau$ the homomorphism defined in \REF{tau}. 
By the definition of $A^{+}(W)^{\circ}$, 
$\alpha^{-1}A^{+}(W)^{\circ}\alpha = A^{+}(W)^{\circ}$ for $\alpha \in A^{+}(W)^{\times}$. 
Hence we have $G^{+}(W) = A^{+}(W)^{\times}$. 
Moreover, 
under the isomorphism \REF{is3} we can understand that 
\begin{gather*}
x\tau(\alpha) = \alpha^{-1}x\xi\alpha\xi^{-1} 
\end{gather*}
for $x \in W$ and $\alpha \in A^{+}(W)^{\times}$. 

Now, 
the pair $(A^{+}(W),\, \nu)$ can be viewed as a quaternary quadratic space over $F$. 
We note that $\nu(x,\, y) = 2^{-1}\mathrm{Tr}_{A^{+}(W)/F}(xy^{*})$ for $x,\, y \in A^{+}(W)$. 
Assume that $F$ is a global or local field. 
For an integral lattice $N$ in $W$ with respect to $\psi$, 
we consider the order $A^{+}(N)$ in $A^{+}(W)$ as defined in the introduction. 
Its discriminant $d(A^{+}(N))$ is given by  
$[A^{+}(N)\widetilde{\,}/A^{+}(N)] = d(A^{+}(N))^{2}$, 
where $A^{+}(N)\widetilde{\,}=\{x\in A^{+}(W)\mid2\nu(x,\, A^{+}(N))\subset\g\}$. 
\begin{lem} \label{l1}
Let $N$ (resp.\ $M$) be an integral lattice (resp.\ a $\g$-maximal lattice) 
in $W$ with respect to $\psi$. 
Then $[A^{+}(M)\widetilde{\,}/A^{+}(M)]$ and $[A^{+}(M)/A^{+}(N)]$ are independent of the choice of $M$. 
Moreover the following assertions hold: 
\begin{enumerate}
\item $[A^{+}(M)/A^{+}(N)] = [M/N]^{2}$. 
\item $[A^{+}(N)\widetilde{\,}/A^{+}(N)] = (2^{-1}[\widetilde{N}/N])^{2}$. 
\item $d(A^{+}(N)) = 2^{-1}[\widetilde{N}/N]$. 
\end{enumerate}
\end{lem}
It should be remarked that we use the same symbol $\widetilde{\ }$ in the sense of 
$\widetilde{N}$ or $A^{+}(N)\widetilde{\,}$ with respect to $\psi$ or $\nu$, respectively. 
We also note that assertion (3) can be verified from Peters \cite[Satz 7]{p}. 
As an application of this lemma it can be seen that 
\textit{if the order $A^{+}(N)$ is maximal in $A^{+}(W)$ for an integral lattice $N$ in $(W,\, \psi)$, 
then $N$ is $\g$-maximal with respect to $\psi$.} 
The converse is not true; 
namely, 
in general, 
$A^{+}(N)$ is not maximal even if $N$ is a maximal lattice. 
\begin{proof}
All assertions in the global case can be obtained from the corresponding local assertions by localization. 
Hence we prove them in the local case. 

Let $M^{\prime}$ be another maximal lattice in $W$. 
Then $M^{\prime}\gamma = M$ with some $\gamma \in SO^{\psi}(W)$. 
By \cite[Lemma 3.8(ii)]{04}, 
$\gamma$ can be extended to an $F$-linear ring-automorphism of $A(W)$. 
Clearly $A^{+}(M^{\prime})\gamma = A^{+}(M^{\prime}\gamma) = A^{+}(M)$. 
It can be seen that $(y\gamma)^{*} = y^{*}\gamma$ for $y \in A(W)$. 
Hence $\nu[y\gamma] = \nu[y]$ 
for every $y \in A^{+}(W)$. 
This implies that the extended $\gamma$ on $A^{+}(W)$ has determinant $\pm 1$. 
Employing this $\gamma$, 
we have $[A^{+}(M)/A^{+}(N)] = [A^{+}(M^{\prime})/A^{+}(N)]$. 
Therefore $[A^{+}(M)/A^{+}(N)]$ is independent of the choice of $M$. 

(1) Let $M_{0}$ be a maximal lattice containing $N$. 
There is a $\g$-basis $\{k_{i}\}_{i=1}^{3}$ of $M_{0}$ such that 
$N = \sum_{i=1}^{3}\g \varepsilon_{i}k_{i}$ with the elementary divisors $\{\varepsilon_{i}\g\}_{i=1}^{3}$. 
Then $[M_{0}/N] = [\g_{3}^{1}/\g_{3}^{1}\varepsilon] = \varepsilon_{1}\varepsilon_{2}\varepsilon_{3}\g$, 
where $\varepsilon = \mathrm{diag}[\varepsilon_{1},\, \varepsilon_{2},\, \varepsilon_{3}]$. 
Furthermore, 
we have 
$[A^{+}(M_{0})/A^{+}(N)] = [\g_{4}^{1}/\g_{4}^{1}\varepsilon^{\prime}] 
	= (\varepsilon_{1}\varepsilon_{2}\varepsilon_{3})^{2}\g = [M_{0}/N]^{2}$, 
where $\varepsilon^{\prime} = \mathrm{diag}[1,\, \varepsilon_{1}\varepsilon_{2},\, 
	\varepsilon_{1}\varepsilon_{3},\, \varepsilon_{2}\varepsilon_{3}]$. 
Because $[A^{+}(M_{0})/A^{+}(N)]$ and $[M_{0}/N]$ are both independent of $M_{0}$, 
we have (1). 

(2) Let $N = \sum_{i=1}^{3}\g e_{i}$. 
With respect to this basis, 
$N$ can be identified with $\g_{3}^{1}$ and 
$\widetilde{N}$ with $\g_{3}^{1}(2\psi_{0})^{-1}$, 
where $\psi_{0}$ is the matrix representing $\psi$ with respect to $\{e_{i}\}_{i=1}^{3}$. 
Then $[\widetilde{N}/N] = [\g_{3}^{1}/\g_{3}^{1}2\psi_{0}] = 2^{3}\det(\psi_{0})\g$. 
On the other hand, 
since $A^{+}(N)$ has a $\g$-basis $\{1,\, e_{i}e_{j}\}_{1\le i<j\le3}$, 
$A^{+}(N)$ can be identified with $\g_{4}^{1}$ and 
$A^{+}(N)\widetilde{\,}$ with $\g_{4}^{1}(2\nu_{0})^{-1}$, 
where $\nu_{0}$ is the matrix representing $\nu$ with respect to that basis. 
Then $[A^{+}(N)\widetilde{\,}/A^{+}(N)] = [\g_{4}^{1}/\g_{4}^{1}2\nu_{0}] = 2^{4}\det(\nu_{0})\g$. 
In view of $e_{i}e_{j} + e_{j}e_{i} = 2\psi(e_{i},\, e_{j})$, 
straightforward calculations show that $\det(\nu_{0}) = \det(\psi_{0})^{2}$. 
This proves (2). 
In particular, 
$[A^{+}(M)\widetilde{\,}/A^{+}(M)] = (2^{-1}[\widetilde{M}/M])^{2}$, 
which does not depend on the choice of maximal $M$. 

(3) Applying \cite[Lemma 2.2(3)]{id} to $A^{+}(N) \subset A^{+}(M_{0})$ with $M_{0}$ in the proof of (1), 
and combining with the results of (1) and (2), 
we have 
\begin{eqnarray*}
[A^{+}(N)\widetilde{\,}/A^{+}(N)] &=& 
	[A^{+}(M_{0})/A^{+}(N)]^{2}[A^{+}(M_{0})\widetilde{\,}/A^{+}(M_{0})] \label{32} \\
	&=& ([M_{0}/N]^{2}\cdot 2^{-1}[\widetilde{M_{0}}/M_{0}])^{2} = (2^{-1}[\widetilde{N}/N])^{2}, \nonumber
\end{eqnarray*}
which proves (3). 
\end{proof}

We shall here recall several results in \cite[Section 5]{6b}, 
which concern \textit{maximal} lattices with respect to ternary forms. 
Let us restate those facts as follows: 
\begin{thm} \label{shi} $\mathrm{(Shimura)}$ 
Let $(W,\, \psi)$ be a ternary quadratic space over a number field $F$ 
and $\xi$ an isomorphism in \REF{is3} of $W$ onto $\awc$ 
such that $\psi[x] = (\xi\xi^{*})^{-1}\nc[x\xi]$ for $x \in W$. 
Put $\xi\xi^{*}\g = \mathfrak{a}\mr^{2}$ 
with a squarefree integral ideal $\mathfrak{a}$ and a $\g$-ideal $\mr$ of $F$. 
Let $M$ be a maximal lattice in $W$. 
Then the following assertions hold: 
\begin{enumerate}
\item There exists a unique order $\mathfrak{O}$ in $\aw$ of 
discriminant $\mathfrak{a}_{0}\mathfrak{e}$ containing $\am$. 
Here $\mathfrak{a}_{0}$ is the product of the prime ideals of $\mathfrak{a}$ 
that are prime to the discriminant $\mathfrak{e}$ of $\aw$. 
For $v \in \mathbf{h}$, 
$\mathfrak{O}_{v}$ is a unique maximal order in the division algebra $\aw_{v}$, 
which is given by $\{x \in \aw_{v} \mid \n_{v}[x] \in \g_{v}\}$ 
if $\mathfrak{a}_{v} = \g_{v}$ and $\psi_{v}$ is anisotropic, 
and $\mathfrak{O}_{v} = \am_{v}$ otherwise. 
\item For $v \in \mathbf{h}$, 
$M_{v}$ can be given by 
\begin{gather}
(M\xi)_{v} = 
\begin{cases}
\mr_{v}(\mathfrak{O}_{v} \cap \awc_{v}) & \text{if $\mathfrak{a}_{v} = \g_{v}$}, \\
\pe_{v}\mr_{v}(\amd_{v} \cap \awc_{v}) & \text{if $\mathfrak{a}_{v} = \pe_{v}$}. 
\end{cases} \label{mrec}
\end{gather}
\item $C(M) = \tau(T(\mathfrak{O}))$ and $\Gamma(M) = \tau(\Gamma^{*}(\mathfrak{O}))$ 
with the notation of \REF{stbl} and \REF{nml}. 
\item The class number of the genus of $M$ with respect to $SO^{\psi}(W)$ equals 
the type number of $\mathfrak{O}$. 
\end{enumerate}
\end{thm}

Assertions (1) and (2) can be seen from \cite[Lemma 5.3(ii)]{6b} and its proof; 
we remark that the constant $d$ in that statement is 
$(\xi\xi^{*})^{-1}$ in the present one and so it must be viewed as 
$(\xi\xi^{*})^{-1}\g = \mathfrak{a}(\mathfrak{a}\mr)^{-2}$. 
Assertions (3) and (4) can be found in Lemma 5.4 and Theorem 5.9 of \cite{6b}, respectively. 

When our lattice $\lw$ in the introduction is maximal, 
\THM{shi} is applicable. 
Notice that the order $\mathfrak{O}$ in (1) contains $\alw$. 
Similarly for non-maximal $\lw$, 
there is an order $\mathfrak{O}$ in $\aw$ containing $\alw$ 
such that $C(\lw) = \tau(T(\mathfrak{O}))$ and $\Gamma(\lw) = \tau(\Gamma^{*}(\mathfrak{O}))$, 
under some conditions on $h$. 
That is given in \cite[Theorem 3.4]{qeo} and relies on our result \REF{rec0}. 

\subsection{Orthogonal complements in quaternary spaces}

Let $(V,\, \p)$ be a four-dimensional quadratic space over a number field $F$. 
The characteristic algebra $Q(\p)$ is determined by $A(\p) \cong M_{2}(Q(\p))$ by definition. 
Put $B = Q(\p)$ and $K = F(\sqrt{\delta})$ with $\delta = \det(\p)$. 
The core dimension $t_{v}$ of $(V,\, \p)$ at $v \in \mathbf{h}$ is determined by 
\begin{gather}
t_{v} = 
\begin{cases}
0 & \text{if $F_{v}(\sqrt{\delta}) = F_{v}$ and $Q(\p_{v}) = M_{2}(F_{v})$}, \\
4 & \text{if $F_{v}(\sqrt{\delta}) = F_{v}$ and $Q(\p_{v})$ is a division algebra}, \\
2 & \text{if $F_{v}(\sqrt{\delta}) \ne F_{v}$}.
\end{cases} \label{cd4}
\end{gather}
This can be seen from \cite[{\S}3.2]{cq} and the proof of \cite[Lemma 3.3]{cq}. 

For $h \in V$ such that $\p[h] = q \ne 0$ we put 
\begin{gather*}
W = (Fh)^{\perp} = \{x \in V \mid \p(x,\, h) = 0\}. 
\end{gather*}
Then $(W,\, \psi)$ is a nondegenerate ternary quadratic space over $F$ 
with the restriction $\psi$ of $\p$ to $W$. 
Clearly $(V,\, \p) = (W,\, \psi) \oplus (Fh,\, \p|_{Fh})$. 
The invariants of $(W,\, \psi)$ are given by 
$\{3,\, F(\sqrt{-\delta q}),\, Q(\psi),\, \{s_{v}(\psi)\}_{v \in \mathbf{r}}\}$, 
which are independent of the choice of $h$. 
The characteristic algebra $Q(\psi) = A^{+}(W)$ is determined by 
the local algebras $Q(\psi_{v})$ for all primes $v$ of $F$. 
It can be seen from \cite[Theorem 7.4\ (i)]{cq} that 
\begin{gather}
M_{2}(Q(\p)) \cong Q(\psi) \otimes_{F} \{K,\ q\}, \label{char}
\end{gather}
where $F$ is a number field or its completion 
and we understand $\{K,\, q\} = M_{2}(F)$ if $K = F$. 
The index at $v \in \mathbf{r}$ is given by $s_{v}(\psi) = s_{v}(\p) - 1$ if $q_{v} > 0$ 
and $s_{v}(\psi) = s_{v}(\p) + 1$ if $q_{v} < 0$. 
Here $q_{v}$ is the image of $q$ under the embedding of $F$ into 
the field $\mathbf{R}$ of real numbers at $v \in \mathbf{r}$. 
The core dimension of $(W,\, \psi)$ at $v \in \mathbf{h}$ is determined by \REF{cd3}. 

The Clifford algebra $A(W)$ can be viewed as a subalgebra of $A(V)$ 
with the restriction $\psi$ of $\p$ to $W$. 
Then $A^{+}(W) = \{x \in A^{+}(V) \mid xh = hx\}$ and 
$G^{+}(W) = \{\alpha \in G^{+}(V) \mid \alpha h = h\alpha\}$ by \cite[Lemma 3.16]{04}. 
The canonical involution of $A(W)$ coincides with $*$ of $A(V)$ restricted to $A(W)$. 
In particular, 
such a $*$ gives the main involution of the quaternion algebra $\aw$. 
\\

Let $L$ and $M$ be $\g$-maximal lattices in $V$ and $W$ with respect to $\p$ and $\psi$, respectively. 
The discriminant ideals of $\p$ and $\psi$ are given as follows: 
\begin{gather}
[\widetilde{L}/L] = D_{K/F}\mathfrak{e}^{2}, \label{d4} \\
[\widetilde{M}/M] = 2\mathfrak{a}^{-1}D_{Q(\psi)}^{2} \cap 2\mathfrak{a}, \label{d3} 
\end{gather}
where $\mathfrak{e}$ is the product of all the prime ideals 
which are ramified in $B$ and which do not ramify in $K$; 
we understand $D_{K/F} = \g$ if $K = F$; 
we put $\delta q\g = \mathfrak{a}\mathfrak{b}^{2}$ 
with a squarefree integral ideal $\mathfrak{a}$ and a $\g$-ideal $\mathfrak{b}$ of $F$. 
These \REF{d4} and \REF{d3} can be seen by applying \cite[Theorem 6.2]{cq} 
to $(V,\, \p)$ and the complement $(W,\, \psi)$. 

The intersection $L \cap W$ is an integral $\g$-lattice in $W$ with respect to $\psi$. 
It can be verified that $[(L \cap W)\widetilde{\,}/L \cap W] = [M/L \cap W]^{2}[\widetilde{M}/M]$ 
and $[M/L \cap W]$ is an integral ideal, 
which is independent of the choice of $M$; see \cite[Lemma 2.2(6)]{id}. 
Moreover there is a $\g$-ideal $\mathfrak{b}(q)$ of $F$ such that 
\begin{gather}
[M/L \cap W] = \mathfrak{b}(q)(2\p(h,\, L))^{-1} \label{42}
\end{gather}
by \cite[Theorem 4.2]{id}. 
We note that $2\p(h,\, L)$ must contain $\mathfrak{b}(q)$ and that $2\p(h,\, L) \subset \g$ if $h \in L$. 
The ideal $\mathfrak{b}(q)$ is determined by \cite[(4.1)]{id} as follows: 
\begin{gather}
2q[\widetilde{L}/L] = \mathfrak{b}(q)^{2}[\widetilde{M}/M] \label{bq}. 
\end{gather}
Now, 
to $L \cap W$ we associate the order $A^{+}(L \cap W)$ as mentioned in the introduction. 
From \LEM{l1}(3) its discriminant is given by 
\begin{gather}
d(A^{+}(L \cap W)) = 2^{-1}[(L \cap W)\widetilde{\,}/L \cap W] = q[\widetilde{L}/L](2\p(h,\, L))^{-2}. 
\label{od}
\end{gather}
We note that if $d(\alw)$ is squarefree, 
then $2\p(h,\, L)$ must be $\bq$ in \REF{42}, 
that is, 
$\lw$ is maximal in $W$. 

\section{The order $\alw$}

Let $F$ be a nonarchimedean local field and $\pe$ the prime ideal of $F$. 
For $b \in F^{\times}$ we set 
\begin{gather*}
(F(\sqrt{b})/\pe) = 
	\begin{cases}
	1 & \text{if $F(\sqrt{b}) = F$}, \\
	-1 & \text{if $F(\sqrt{b})$ is an unramified quadratic extension of $F$}, \\
	0 & \text{if $F(\sqrt{b})$ is a ramified quadratic extension of $F$}. 
	\end{cases}
\end{gather*}
For a quaternion algebra $B$ over $F$ we also set 
\begin{gather*}
\chi(B) = 
	\begin{cases}
	1 & \text{if $B \cong M_{2}(F)$}, \\
	-1 & \text{if $B$ is a division algebra}. 
	\end{cases}
\end{gather*}

Let $(V,\, \p)$ be a quaternary quadratic space over $F$. 
For $h \in V$ such that $\p[h] = q \ne 0$, 
put $W = (Fh)^{\perp}$ and let $\psi$ be the restriction of $\p$ to $W$. 
Let $t$ be the core dimension of $\p$. 

\begin{lem} \label{l0} 
Suppose that $\p$ is isotropic. 
Let $L$ be a $\g$-maximal lattice in $V$ with respect to $\p$. 
Put $q\g = q_{0}\pe^{2\ell}$ with $q_{0} \in \g^{\times} \cup \pi\g^{\times}$ and $\ell \in \mathbf{Z}$. 
If $\mathfrak{b}(q) = \pe^{\ell}$, 
then $A^{+}(L \cap W) = A^{+}(L) \cap A^{+}(W)$. 
\end{lem}
\begin{proof}
Since $\p$ is isotropic, 
$t$ is $0$ or $2$. 
Let $K$ be the discriminant algebra of $\p$ defined by 
$K = F \times F$ or $K = F(\sqrt{\det(\p)})$ according as $t = 0$ or $t = 2$. 
Because $K$ is embeddable in $A^{+}(V)$, 
we identify $K$ with the image of it. 
Then there is a \textit{weak} Witt decomposition as follows: 
\begin{gather}
V = Kg \oplus (Fe + Ff), \qquad L = \mathfrak{r}g + (\g e + \g f), \nonumber \\ 
(Kg,\ \p) \cong (K,\ c\kappa) \quad \text{via } xg \longmapsto x \label{w02} 
\end{gather}
with some elements $e$ and $f$ of $V$ such that $\p[e] = \p[f] = 0$ and $2\p(e,\, f) = 1$, 
and also $g \in V$ such that $g^{2} = c \in F^{\times}$. 
Here $\mathfrak{r} = \g \times \g$ if $t = 0$ and 
$\mathfrak{r}$ is the maximal order of $K$ if $t = 2$. 
We may assume that $c = 1$ if $t = 0$, 
$c \in \g^{\times}$ if $(K/\pe) = -1$ and $\chi(Q(\p)) = +1$, 
$c \in \pi \g^{\times}$ if $(K/\pe) = -1$ and $\chi(Q(\p)) = -1$, 
and $c \in \g^{\times}$ if $(K/\pe) = 0$. 
We note that $A^{+}(Kg) = K$ and $xg = gx^{*}$ for $x \in K$, 
and that $(a,\, b)^{*} = (b,\, a)$ and $\kappa[(a,\, b)] = ab$ for $(a,\, b) \in K$ when $t = 0$. 

Put $2\p(h,\, L) = \pe^{m}$. 
Then $m \le \ell$ because $\mathfrak{b}(q)$ is contained in $2\p(h,\, L)$. 
We take $k = q\pi^{-m}e + \pi^{m}f$. 
This satisfies $\p[h] = \p[k]$ and $2\p(h,\, L) = 2\p(k,\, L)$. 
By virtue of \REF{th13}, 
there exists $\alpha \in C(L)$ such that $h\alpha = k$. 
Under the isomorphism $\alpha$, 
we may identify $W$ with $(Fk)^{\perp}$ 
and $\psi$ with the restriction of $\p$ to $(Fk)^{\perp}$; 
for details, see the explanation of \REF{idrc} below. 
Then we have 
\begin{gather}
W = Kg \oplus F(q\pi^{-m}e - \pi^{m}f), \qquad L \cap W = \mathfrak{r}g + \pe^{-m}(q\pi^{-m}e - \pi^{m}f). 
\label{wl0}
\end{gather}
We set $M = \mathfrak{r}g + \pe^{-\ell}(q\pi^{-m}e - \pi^{m}f)$. 
Since $\psi[M] \subset \g$ and $[M/L \cap W] = [\pe^{-\ell}/\pe^{-m}] = \mathfrak{b}(q)(2\p(h,\, L))^{-1}$, 
$M$ is a $\g$-maximal lattice in $(W, \, \psi)$ containing $L \cap W$. 

We are going to show that $A^{+}(M)$ contains $A^{+}(L) \cap A^{+}(W)$ 
by means of the $F$-linear ring-isomorphism $\Psi$ of $A(V)$ introduced in \cite[{\S}2.4]{04}. 
Recall that $\Psi(A(V)) = M_{2}(A(Kg)) = M_{2}(\{K,\, g^{2}\})$ and 
\begin{gather*}
\Psi(A^{+}(V)) = \left \{
\begin{bmatrix}
   		a & b \\ 
		c & d \\
\end{bmatrix}
\in M_{2}(A(Kg)) \Biggm| a,\, d \in K,\ b,\, c \in Kg
\right \}. 
\end{gather*}
Since $A^{+}(W) = \{x \in A^{+}(V) \mid xh = hx\}$ and 
$\Psi(h) = 
\left[
  	\begin{array}{@{\, }cc@{\, }}
   		0 & q\pi^{-m} \\ 
		\pi^{m} & 0 
  	\end{array} 
\right] 
$, 
we see that 
\begin{gather}
\Psi(A^{+}(W)) = \{x \in \Psi(A^{+}(V)) \mid x\Psi(h) = \Psi(h)x\} = K + K\eta, \label{l0aw} \\ 
\eta = 
\begin{bmatrix}
   		0 & q\pi^{-m}g \\ 
		\pi^{m}g & 0 \\
\end{bmatrix}. 
\label{eta}
\end{gather}
Here we identify $K$ with $\Psi(K) = \{\mathrm{diag}[a,\, a] \mid a \in K\}$. 
Put $\mathfrak{r} = \g[\xi]$. 
We have then 
\begin{gather*}
A^{+}(M) = \mathfrak{f} + \pi^{-\ell}\mathfrak{r}g(q\pi^{-m}e - \pi^{m}f) 
	\cong \Psi(A^{+}(M)) = \mathfrak{f} + \pi^{-\ell}\mathfrak{r}\eta, 
\end{gather*}
where $\mathfrak{f} = \g + c\g \xi$, the order in $K$ of conductor $c\g$. 
Now, 
$\Psi(A(L)) = M_{2}(A(\mathfrak{r}g))$. 
This combined with $A^{+}(L) \cap A^{+}(W) = A(L) \cap A^{+}(W)$ and 
$A(\mathfrak{r}g) = \mathfrak{f} + \mathfrak{r}g$ shows that 
\begin{eqnarray}
\Psi(A^{+}(L) \cap A^{+}(W)) &=& 
\left \{
\begin{bmatrix}
   		a & bq\pi^{-m}g \\ 
		b\pi^{m}g & a \\
\end{bmatrix}
\in M_{2}(\mathfrak{f} + \mathfrak{r}g) 
\biggm| a,\, b \in K
\right \} \nonumber \\
&=& \mathfrak{f} + \pi^{-m}\mathfrak{r}\eta. \label{l0alw}
\end{eqnarray}
Hence $A^{+}(M)$ contains $A^{+}(L) \cap A^{+}(W)$. 
Moreover we have 
$[A^{+}(M)/A^{+}(L) \cap A^{+}(W)] = 
	[\pe^{-\ell}\mathfrak{r}/\pe^{-m}\mathfrak{r}] = [M/L \cap W]^{2}$. 
On the other hand, 
$A^{+}(L) \cap A^{+}(W)$ contains $A^{+}(L \cap W)$ and 
$[A^{+}(M)/A^{+}(L \cap W)] = [M/L \cap W]^{2}$ by \LEM{l1}(1). 
Thus $A^{+}(L \cap W)$ must coincide with $A^{+}(L) \cap A^{+}(W)$. 
\end{proof}

\begin{lem} \label{lalw}
Let $L$ be a $\g$-maximal lattice in $V$ with respect to $\p$. 
Then $A^{+}(L \cap W) = A^{+}(L) \cap A^{+}(W)$ 
for every $h \in V$ such that $\p[h] \ne 0$. 
\end{lem}

Let $K = F(\sqd)$ be the discriminant field of $\p$ and put $B = Q(\p)$. 
In the rest of this section, 
we shall prove \LEM{lalw} according to the core dimension $t = 0,\, 2,\, 4$ of $(V,\, \p)$. 
To avoid a repetition of the same argument, 
the proof will be omitted in several such cases. 

First of all, 
we may assume that $h \in L$. 
In fact, 
$(W,\, \psi)$ is determined by $F^{\times}h$, 
so that, 
replacing $h$ by $ch$ with a suitable element $0 \ne c \in \g$, 
we have $h \in L$. 

Let $q = q_{0}\pi^{2\ell}$ with $q_{0} \in \g^{\times} \cup \pi\g^{\times}$ and $0 \le \ell \in \mathbf{Z}$. 
Put $2\p(h,\, L) = \pe^{m}$ with $0 \le m \le \lambda$, 
where $\mathfrak{b}(q) = \pe^{\lambda}$. 

\subsection{Case $t = 0$} 

In this case $K = F$ and $B = M_{2}(F)$. 
By \REF{char}, 
$Q(\psi) = M_{2}(F)$ and so the core dimension of $W$ is $1$ by \REF{cd3}. 
The localizations of \REF{d4} and \REF{d3} show that 
$[\widetilde{L}/L] = \g$ and $[\widetilde{M}/M] = 2q_{0}\g$. 
We have then $\mathfrak{b}(q) = \pe^{\ell}$ by \REF{bq}. 
Hence \LEM{l0} proves the assertion in this case. 

\subsection{Case $t = 2$ and $(K/\pe) = -1$} 

We first recall by \REF{cd4} that $t = 2$ if and only if $K \ne F$. 
Assume that $K$ is unramified over $F$. 
Then $\delta \in \g^{\times}$ and $D_{K/F} = \g$. 
Put $\p[h] = q = \varepsilon\pi^{\nu}$ with $\varepsilon \in \g^{\times}$ and $0 \le \nu \in \mathbf{Z}$. 
We take a Witt decomposition of $\p$ as in \REF{w02}. 
Since the isomorphism class of $(Kg,\, \p)$ is determined by $K$ and $c\kappa[K^{\times}]$, 
we may assume with suitably chosen $g \in V$ that 
$c \in \g^{\times}$ or $c \in \pi \g^{\times}$ according as $\pe \nmid D_{B}$ or $\pe \mid D_{B}$. 

Suppose $q \in \p[Kg] = c\kappa[K]$. 
Then $Q(\psi) = M_{2}(F)$ by applying \cite[Theorem 1.1]{id}; 
the core dimension of $W$ is $1$ by \REF{cd3}. 

If $\nu$ is even, 
then $q_{0} \in \g^{\times}$ and $\pe \nmid D_{B}$. 
We have $2q[\widetilde{L}/L][\widetilde{M}/M]^{-1} = \pe^{2\ell}$ by \REF{d4} and \REF{d3}, 
and so $\mathfrak{b}(q) = \pe^{\ell}$. 
Hence \LEM{l0} is applicable to this case. 

If $\nu$ is odd, 
then $B$ is a division algebra. 
We are going to show that 
\begin{gather*}
[A^{+}(M)/A^{+}(L) \cap A^{+}(W)] = [M/L \cap W]^{2}. 
\end{gather*} 
We may take $c = \pi\e$ as $K$ is unramified over $F$. 
There is $z = \pi^{\ell} \in \mr$ so that $\p[zg] = q$. 
Then, 
by virtue of \cite[Theorem 3.5]{y}, 
we can find $k = zg + \pi^{m}e \in L$ so that $\p[k] = \p[h]$ and $\p(k,\, L) = \p(h,\, L)$. 
Thus \REF{th13} provides an element $\gamma \in SO^{\p}$ 
such that $h\gamma = k$ and $L\gamma = L$. 
Under the isomorphism $\gamma$ we may identify $h$ with $k$, 
$W$ with $(Fk)^{\perp}$, 
and $\psi$ with the restriction of $\p$ to $(Fk)^{\perp}$, respectively. 
Then \cite[Lemma 2.4(3)]{id} gives a Witt decomposition of $(W,\, \psi)$ and a maximal lattice $M$ in $W$ as follows: 
\begin{gather}
W = F\sqrt{\delta}g \oplus (Fe + Ff_{0}),\qquad M = \g \sqrt{\delta}g + \g e + \g f_{0}, \label{t2-11} \\
f_{0} = f - \frac{\pi^{2m}}{4q}e - \frac{\pi^{m}}{2q}zg. \label{f0}
\end{gather}

If $\pe \nmid 2$, 
then $\mr = \g[\sqd]$. 
We have 
\begin{eqnarray}
A^{+}(L) &=& \g + \g \pi\sqrt{\delta} + \g ge + \g gf + \g \sqd ge + \g \sqd gf + \g ef + \g \pi\sqd ef, 
	\nonumber \\
A^{+}(M) &=& \g + \g \sqd ge + \g \sqd gf_{0} + \g ef_{0}. \label{amt2-1}
\end{eqnarray}
Observe that 
$\sqd gf_{0} = \sqd gf - (4q)^{-1}\pi^{2m}\sqd ge - (2q)^{-1}\pi^{\ell+m+1}\e \sqd$ 
and $ef_{0} = ef + (2q)^{-1}\pi^{\ell+m}ge$. 
For $x = x_{1} + x_{2}\sqd ge + x_{3}\sqd gf_{0} + x_{4}ef_{0} \in A^{+}(W)$, 
we see that 
\begin{eqnarray*}
x \in A^{+}(L) &\Longleftrightarrow& 
	\begin{cases}
	x_{1} \in \g, & x_{2} - x_{3}(4q)^{-1}\pi^{2m} \in \g, \\ 
	x_{3} \in \g \cap 2q\pi^{-\ell-m}\g, & 
	x_{4} \in \g \cap 2q\pi^{-\ell-m}\g  
	\end{cases} \\ 
&\Longleftrightarrow& 
	[x_{1}\ x_{2}\ x_{3}\ x_{4}] \in \g^{1}_{4}\alpha_{0}, 
\end{eqnarray*}
where $\alpha_{0}$ is a triangular matrix in $GL_{4}(F)$ given by 
\begin{gather}
\begin{bmatrix}
   		1 & 0 & 0 & 0 \\ 
   		0 & 1 & 0 & 0 \\ 
   		0 & (4q)^{-1}\pi^{\ell+m+1} & \pi^{\ell+1-m} & 0 \\ 
   		0 & 0 & 0 & \pi^{\ell+1-m} \\
\end{bmatrix}. 
\label{alw2-111}
\end{gather} 
Thus the $F$-linear automorphism $\alpha$ of $A^{+}(W)$, 
represented by $\alpha_{0}$ with respect to the basis 
$\{1,\, \sqd ge,\, \sqd gf_{0},\, ef_{0}\}$ of $A^{+}(M)$, 
gives a surjection of $A^{+}(M)$ onto $A^{+}(L) \cap A^{+}(W)$. 
Therefore by \cite[Lemma 2.2(1)]{id} we find that 
$[A^{+}(M)/A^{+}(L) \cap A^{+}(W)] = \pe^{2(\ell+1-m)}$. 
Since $\bq = \pe^{\ell+1}$, 
this equals $[M/L \cap W]^{2}$ by \REF{42}. 
On the other hand, 
$A^{+}(L) \cap A^{+}(W)$ contains $A^{+}(L \cap W)$ and 
$[A^{+}(M)/A^{+}(L \cap W)] = [M/L \cap W]^{2}$ by \LEM{l1}(1). 
Thus $A^{+}(L) \cap A^{+}(W)$ must coincide with $A^{+}(L \cap W)$. 

If $\pe \mid 2$, 
by \cite[Lemma 3.5(2)]{99b}, 
we can put $\delta = (1 + \pi^{2\kappa}a)b^{-2}$ with $a,\, b \in \gx$, 
where $2\g = \pe^{\kappa}$. 
Also put $u = 2^{-1}(1 + b\sqd)$; 
then $\mr = \g[u] = \g \sqd + \g u$ and $\sqd = b\delta + 2\sqd u^{*}$. 
By using these, 
we can verified in a similar way to the case $\pe \nmid 2$ 
that there exists an $F$-linear automorphism $\alpha$ of $A^{+}(W)$ 
such that $A^{+}(M)\alpha = A^{+}(L) \cap A^{+}(W)$, 
which is represented by 
\begin{gather}
\begin{bmatrix}
   		1 & 0 & 0 & 0 \\ 
   		0 & 1 & 0 & 0 \\ 
   		2^{-1}\pi b\dl \e & (4q)^{-1}\pi^{\ell+m+1} & \pi^{\ell+1-m} & 0 \\ 
   		0 & (2\e)^{-1}b & 0 & \pi^{\ell+1-m} \\
\end{bmatrix} 
\label{310}
\end{gather}
with respect to the basis $\{1,\, \sqd ge,\, \sqd gf_{0},\, ef_{0}\}$ 
of $A^{+}(M)$ given in \REF{amt2-1}. 
Observing $[A^{+}(M)/A^{+}(L) \cap A^{+}(W)] = [M/L \cap W]^{2}$, 
we have $A^{+}(L)\cap A^{+}(W) = A^{+}(L \cap W)$. 

As for the case $q \not\in \p[Kg]$, 
the discriminant ideals of $\p$ and $\psi$ are respectively given by 
$[\widetilde{L}/L] = \pe^{2}$ or $\g$ and $[\widetilde{M}/M] = 2\pe^{2}$ or $2\pe$ 
according as $\n$ is even or $\n$ is odd. 
From these we have $\bq = \pe^{\ell}$ and so we can apply \LEM{l0}. 

\subsection{Case $t = 2$ and $(K/\pe) = 0$}

Our aim is to construct a surjection of $A^{+}(M)$ onto $A^{+}(L) \cap A^{+}(W)$ 
under the setting of \cite[{\S}3.2]{id}. 
We take a Witt decomposition of $\p$ in \REF{w02} with $g \in V$ so that $c = g^{2} \in \gx$. 
Notice that the isomorphism class of $(Kg,\, \p)$ is determined by $K$ and $c\kappa[K^{\times}]$. 
Put $2\g = \pe^{\kappa}$. 

Suppose $\dl \in \gx$. 
Then $\pe \mid 2$. 
By \cite[Lemma 3.5]{99b}, 
we can put $D_{K/F} = \pe^{2(\kappa-k)}$ and $\dl = (1 + \pi^{2k+1}a)b^{-2}$ 
with $a,\, b \in \gx$ and $0 \le k < \kappa$. 
Also put $u = \pi^{-k}(1 + b\sqd) \in \mr$;
then $\mr = \g[u]$ and $\ka[u] = -\pi a$, 
which is a prime element of $F$. 
We set $q = (-\pi a)^{\n}\e$ with $\e \in \gx$. 

If $q \in \p[Kg]$, 
then our assertion can be seen in a similar way to the case where $(K/\pe) = -1$ and $q \in \p[Kg]$. 

If $q \not\in \p[Kg]$, 
then $\psi$ is anisotropic. 
Fix an element $s \in 1 + \pe^{2(\ka-k)-1}$ so that $s \not\in \ka[\mr^{\times}]$. 
Because $q \in \e \ka[K^{\times}]$, 
we may take $g$ in \REF{w02} so that $g^{2} = s\e$. 
By \cite[Theorem 3.5]{y} combined with \REF{th13}, 
we can identify $h$ with $h^{\prime} = zg + \pi^{-m}q(1 - s)e + \pi^{m}f$ 
and $W$ with $(Fh^{\prime})^{\perp}$ for a fixed $z \in \mr$ so that $\p[zg] = sq$. 
In \cite[(3.2)]{id}, $(W,\, \psi)$ can be given by 
\begin{gather} 
W = ((Fzg)^{\perp} \cap Kg) \oplus (Fg_{2} + Fg_{3}), \label{t2003} \\
g_{2} = \frac{q(s-1)}{\pi^{2m}}e + f, \qquad 
g_{3} = e - \frac{\pi^{m}}{2sq}zg. \label{g2g3}
\end{gather}

Let $\n$ be even. 
We take $z = (-\pi a)^{\ell}$; 
then $\p[zg] = sq$ and $W = F\sqd g \oplus (Fg_{2} + Fg_{3})$; 
\cite[(3.6)]{id} gives a maximal lattice $M$ in $W$ as follows: 
\begin{gather}
M = \g \sqd g + \g \pi^{-\mm}g_{2} + \g g_{0}, \qquad 
g_{0} = \pi^{-k}\sqd g + \frac{2s\e a^{\ell}}{\pi^{\ka+k+\mmm}b}g_{3} \label{m20030}, 
\end{gather}
where $\la = \ell+\ka-k-1$, $\mm = \la - m$, and $\mmm = m-\ka-\ell$. 
Noticing $L = (\g \sqd + \g u)g + (\g g_{2} + \g e)$, 
we have 
\begin{eqnarray}
A^{+}(L) &=& \g + \g \sqd u^{*} + \g \sqd gg_{2} + \g \sqd ge + \nonumber \\
	& & \qquad + \g ugg_{2} + \g uge + \g g_{2}e + \g \sqd u^{*}g_{2}e, \nonumber \\
A^{+}(M) &=& \g + \g \pi^{-\mm}\sqd gg_{2} + \g \sqd gg_{0} + \g \pi^{-\mm}g_{2}g_{0}. \label{am20030}
\end{eqnarray}
Straightforward calculations show that 
\begin{eqnarray*}
\sqd gg_{0} &=& -\frac{\dl s\e(1 + (-1)^{\ell})}{\pi^{k}} + 
	\frac{2s\e a^{\ell}}{\pi^{\ka+k+\mmm}b}\sqd ge - 
	\frac{(-1)^{\ell+k}a^{k}s\e}{b}\sqd u^{*}, \\
g_{2}g_{0} &=& \frac{\dl s\e(1 + (-1)^{\ell})}{\pi^{k}}\sqd gg_{2} + 
	\frac{2s\e a^{\ell}}{\pi^{\ka+k+\mmm}b}g_{2}e - 
	\frac{(-1)^{\ell+k}a^{k}}{b}ugg_{2}. 
\end{eqnarray*}
For $x = x_{1} + x_{2}\pi^{-\mm}\sqd gg_{2} + x_{3}\sqd gg_{0} + x_{4}\pi^{-\mm}g_{2}g_{0} \in A^{+}(W)$, 
we see that 
\begin{eqnarray*}
x \in A^{+}(L) &\Longleftrightarrow& 
	\begin{cases}
	x_{1} - x_{3}\pi^{-k}\dl s\e(1 + (-1)^{\ell}) \in \g, & \\ 
	x_{2} \in \pi^{\mm}\g, & \\ 
	x_{3} \in \g \cap (2s\e a^{\ell})^{-1}\pi^{\ka+k+\mmm}b\g, & \\ 
	x_{4} \in \pi^{\mm}(\g \cap (2s\e a^{\ell})^{-1}\pi^{\ka+k+\mmm}b\g) 
	\end{cases} \\
&\Longleftrightarrow& 
	[x_{1}\ x_{2}\ x_{3}\ x_{4}] \in \g^{1}_{4}\alpha_{0}, 
\end{eqnarray*}
where $\alpha_{0} \in GL_{4}(F)$ is given by 
\begin{gather}
\begin{bmatrix}
   		1 & 0 & 0 & 0 \\ 
   		0 & \pi^{\mm} & 0 & 0 \\ 
   		\pi^{-k}\dl s\e(1 + (-1)^{\ell}) & 0 & 1 & 0 \\ 
   		0 & 0 & 0 & \pi^{\mm} \\
\end{bmatrix}. 
\label{alw20030}
\end{gather}
Thus the $F$-linear automorphism $\alpha$ of $A^{+}(W)$, 
represented by $\alpha_{0}$ with respect to the basis 
$\{1,\, \pi^{-\mm}\sqd gg_{2},\, \sqd gg_{0},\, \pi^{-\mm}g_{2}g_{0}\}$ of $A^{+}(M)$, 
gives a surjection of $A^{+}(M)$ onto $A^{+}(L) \cap A^{+}(W)$. 
Therefore, 
together with $\bq = \pe^{\la}$, 
we obtain 
$[A^{+}(M)/A^{+}(L) \cap A^{+}(W)] = \pe^{2(\la-m)} = [M/L \cap W]^{2}$. 
Thus $A^{+}(L) \cap A^{+}(W)$ must be $A^{+}(L \cap W)$. 

Let $\n$ be odd. 
We take $z = (-\pi a)^{\ell}u$. 
Observe that $\p[zg] = sq$ and $W = F(b\dl + \sqd)g \oplus (Fg_{2} + Fg_{3})$. 
Put $v = \pi^{-k}(b\dl + \sqd) \in \mr$. 
By \cite[(3.6)]{id} we have a maximal lattice $M$ in $W$ as follows: 
\begin{gather}
M = \g vg + \g \pi^{-\mm}g_{2} + \g g_{0}, \qquad 
g_{0} = \pi^{-k-1}vg + \frac{2s\e a^{\ell+1}}{\pi^{\ka+k+\mmm+1}b}g_{3}, \label{320}
\end{gather}
where $\la = \ell+\ka-k$, $\mm = \la - m$, and $\mmm = m-\ka-\ell-1$. 
Noticing $L = (\g + \g v)g + (\g g_{2} + \g e)$, 
we have 
\begin{eqnarray*}
A^{+}(L) &=& \g + \g v + \g vgg_{2} + \g vge + \g gg_{2} + \g ge + \g g_{2}e + \g vg_{2}e, \nonumber \\
A^{+}(M) &=& \g + \g \pi^{-\mm}vgg_{2} + \g vgg_{0} + \g \pi^{-\mm}g_{2}g_{0}. \label{am20031}
\end{eqnarray*}
Then the desired fact can be derived in the same way as in the case of even $\n$. 

Suppose $\dl \in \pi\gx$. 
We may put $q = (-\dl)^{\n}\e$ with $\e \in \gx$ because $\dl$ is a prime element of $F$. 
We have $\mr = \g[\sqd]$ and so $L = (\g + \g \sqd)g + (\g e + \g f)$ in \REF{w02}. 
Then our proof is divided into the two cases of whether or not $q \in \p[Kg]$, 
as seen in the previous case $\dl \in \gx$. 
In either case our assertion can be proven in a similar way. 

\subsection{Case $t = 4$} 

In this case $K = F$ and $B$ is a division algebra; 
$L$ is a unique maximal lattice in $V$ because $\p$ is anisotropic. 
Clearly $(W,\, \psi)$ is anisotropic and so 
$Q(\psi)$ is a division algebra by \REF{char}. 
Also $L \cap W$ is a unique maximal lattice in $W$. 
The discriminant of $A^{+}(L \cap W)$ is 
$2^{-1}[(L \cap W)\widetilde{\,}/L \cap W] = q_{0}^{-1}\pe^{2}$ 
by \LEM{l1}(3) and \REF{d3}. 

If $q_{0} \in \pi\g^{\times}$, 
then $d(A^{+}(L \cap W)) = \pe$, 
whence $A^{+}(L \cap W)$ is a unique maximal order 
$\{x \in A^{+}(W) \mid \nu[x] \in \g\}$ in $A^{+}(W)$, 
where $\nu$ is the norm form of $A^{+}(W)$. 
Since $A^{+}(L \cap W) \subset A^{+}(L) \cap A^{+}(W)$, 
both must be the same order. 

If $q_{0} \in \g^{\times}$, 
we set $K_{1} = F(\sqrt{-q_{0}})$, 
which is the discriminant field of $\psi$. 
We are going to take $\g$-bases of $L$ and $L \cap W$ treated in \cite[{\S}3.4]{id}. 
To do this, 
let us recall the setting given there. 
Put $2\g = \pe^{\kappa}$. 

Assume $K_{1} \ne F$. 
Let $\kappa_{1}$ be the norm form of $K_{1}$ and 
$\mathfrak{r}_{1}$ its maximal order. 
Take an element $c \in \g^{\times}$ so that $c \not\in \kappa_{1}[K_{1}^{\times}]$ 
or $c = \pi$ according as $(K_{1}/\pe) = 0$ or $-1$. 
Then $K_{1} + K_{1}\omega$ is a division quaternion algebra over $F$ with $\omega^{2} = c$ by \REF{om}, 
which can be identified with $B$. 
We may further identify $(V,\, \p)$ with $(B,\, \beta)$, 
where $\beta$ is the norm form of $B$ 
(cf.\ \cite[Theorem 7.5(ii)]{04}). 
Taking $z = \sqrt{-q}$, 
we have $\beta[z] = q = \p[h]$. 
Thus $h\gamma = z$ with some $\gamma \in SO^{\p}$. 
In view of $C(L) = SO^{\p}$ as $\p$ is anisotropic, 
$W$ may be identified with $(Fz)^{\perp}$. 
Then we see that 
\begin{gather}
W = F 1_{B} \oplus K_{1}\omega. \label{w4}
\end{gather}
Here $1_{B}$ is the identity element of $B$. 
We remark that $1_{B}$ should not be confused with the identity element $1$ of $A(V) = A(B)$. 

If $(K_{1}/\pe) = 0$, 
then $\pe \mid 2$. 
We can take $c \in 1 + \pe^{2(\kappa-k)-1}$ so that $c \not\in \kappa_{1}[\mathfrak{r}_{1}^{\times}]$, 
where $D_{K_{1}/F} = \pe^{2(\kappa-k)}$ with $0 \le k < \kappa$. 
Further we can put $-q_{0} = (1 + \pi^{2k+1}a)b^{-2}$ with some $a,\, b \in \g^{\times}$. 
Put $v = \sqrt{-q_{0}}$ and $u = \pi^{-k}(1 + bv)$; 
then $\kappa_{1}[u] = -\pi a$ and so $\mathfrak{r}_{1} = \g[u]$. 
Under this setting we have $\g$-bases of $L$ and $L \cap W$ in \cite[{\S}3.4,\, (3.11)]{id} as follows: 
\begin{eqnarray}
L &=& \g 1_{B} + \g u\omega + 
	\g \frac{\ib+\omega}{\pi^{\kappa-k-1}} + \g \frac{\ib+bv+\omega+bv\omega}{\pi^{\kappa}}, \nonumber \\
L \cap W &=& \g 1_{B} + \g u\omega + \g \frac{\ib+\omega}{\pi^{\kappa-k-1}}. \label{lw40}
\end{eqnarray}
Putting $L = \sum_{i=1}^{4}\g g_{i}$, 
we easily see that 
\begin{gather}
A^{+}(L) \cap A^{+}(W) = \g + \g g_{1}g_{2} + \g g_{1}g_{3} + \g g_{2}g_{3} = A^{+}(L \cap W). \label{alw40}
\end{gather}

If $(K_{1}/\pe) = -1$, 
then $L = \mr_{1} + \mr_{1}\omega$, 
which contains $L \cap W = \g + \mr_{1}\omega$ with $\omega^{2} = \pi$. 
Our assertion in this case can be seen in a straightforward way. 

Assume $K_{1} = F$. 
We may identify $(V,\, \p)$ with $(B,\, -\beta)$ with $\beta$ as above. 
Set $B = J + J\omega$ with an unramified quadratic extension $J$ over $F$ and $\omega^{2} = \pi$. 
There is $z = \sq \in F^{\times}$ such that $-\beta[z] = q$. 
For the same reason as in the case $K_{1} \ne F$, 
we may identify $W$ with $(Fz)^{\perp} = F\sqrt{s} + J\omega$ 
with an element $s \in \g^{\times}$ so that $J = F(\sqrt{s})$. 
Let $\mr_{0}$ be the maximal order of $J$; 
then $L = \mr_{0} + \mr_{0}\omega \supset L \cap W = \g\sqrt{s} + \mr_{0}\omega$. 
If $\pe \mid 2$, 
we may take $s \in 1 + \pi^{2\kappa}\g^{\times}$ so that $s \not\in \g^{\times 2}$; 
putting $u = \pi^{-\kappa}(1 + \sqrt{s})$, 
we have $\mr_{0} = \g[u]$. 
If $\pe \nmid 2$, 
then $\mr_{0} = \g[\sqrt{s}]$. 
Hence 
\begin{gather}
L = \g 1_{B} + \g v + \g \omega + \g v\omega, \qquad 
L \cap W = \g \sqrt{s} + \g \omega + \g v\omega, \label{lw41}
\end{gather}
where $v = u$ or $v = \sqrt{s}$ according as $\pe \mid 2$ or $\pe \nmid 2$. 
Using these bases we have 
\begin{gather}
A^{+}(L \cap W) = 
\g + \g \sqrt{s}\omega + \g \sqrt{s}(v\omega) + \g \omega(v\omega), \label{alw41}
\end{gather}
which is also a $\g$-basis of $A^{+}(L) \cap A^{+}(W)$. 
This completes the proof of \LEM{lalw}. 

\section{Theorem on $\lw$ in the local case}

\begin{thm} \label{th2}
Let $(V,\, \p)$ be a quaternary quadratic space over a nonarchimedean local field $F$. 
For $h \in V$ such that $\p[h] = q \ne 0$, 
put $W = (Fh)^{\perp}$ and let $\psi$ be the restriction of $\p$ to $W$. 
Take an isomorphism $\xi$ as stated in \REF{is3} so that $\xi \xi^{*} \in \g^{\times} \cup \pi\g^{\times}$. 
Let $K = F(\sqd)$ be the discriminant field of $\p$ and put $\delta q\g = \mathfrak{a}\mathfrak{b}^{2}$ 
with $\mathfrak{a} = \g$ or $\mathfrak{a} = \pe$ and a $\g$-ideal $\mathfrak{b}$ of $F$. 
Then, 
for a $\g$-maximal lattice $L$ in $V$ with respect to $\p$, 
$L \cap W$ can be given by 
\begin{gather}
(L \cap W)\xi = \mathfrak{c}[M/L \cap W]\cdot
	\{A^{+}(L \cap W)\widetilde{\,} \cap A^{+}(W)^{\circ}\}. \label{lrec}
\end{gather}
Here $\mathfrak{c} = \mathfrak{a}$ or $\mathfrak{c} = \pe$ according as $Q(\psi) = M_{2}(F)$ or $Q(\psi) \ne M_{2}(F)$ 
and $M$ is a $\g$-maximal lattice in $W$ with respect to $\psi$. 
\end{thm}

We devote {\S}3.1-3.4 to the proof. 
In several cases presented below, 
the proof will be omitted to avoid a repetition of the same argument. 

In the proof of \THM{th2}, 
we may assume that $h \in L$ by a suitable constant multiple. 
We often identify $h$ with a specified element $k$ and $W$ 
with the complement $W^{\prime} = (Fk)^{\perp}$ under some $\alpha \in C(L)$. 
Then it is fundamental for our argument that: 
\begin{gather}
\textit{If \REF{lrec} is valid for $L \cap W^{\prime}$, 
then it is true for $L \cap W$.} \label{idrc} 
\end{gather}
In fact, 
by \cite[Lemma 3.8(ii)]{04}, 
such an $\alpha$ can be extended to 
an $F$-linear ring-automorphism of $A(V)$. 
In view of $(x\alpha)^{*} = x^{*}\alpha$ for $x \in A(V)$, 
we see that $(L \cap W)\alpha = L \cap W^{\prime}$, 
$A^{+}(L)\alpha = A^{+}(L)$, 
$A^{+}(W)\alpha = A^{+}(W^{\prime})$, 
$A^{+}(W)^{\circ}\alpha = A^{+}(W^{\prime})^{\circ}$, 
and $A^{+}(L \cap W)\widetilde{\ }\alpha = A^{+}(L \cap W^{\prime})\widetilde{\,}$ 
with respect to the norm forms of $A^{+}(W)$ and $A^{+}(W^{\prime})$. 
For an element $\xi$ as in \REF{is3} put $\zeta = \xi\alpha$, 
which gives an isomorphism of $W^{\prime}$ onto $A^{+}(W^{\prime})^{\circ}$. 
Clearly $\xi\xi^{*} = \zeta\zeta^{*}$. 
Then it can be seen that $x\alpha\zeta = x\xi\alpha$ for $x \in W$. 
From this fact we can verify \REF{idrc}. 
\\

Hereafter until the end of the proof we set 
\begin{gather*}
\p[h] = q \in \pi^{\nu}\gx, \qquad 2\p(h,\, L) = \pe^{m} 
\end{gather*}
with $0 \le \nu \in \mathbf{Z}$ and $0 \le m \le \lambda$, 
where $\mathfrak{b}(q) = \pe^{\lambda}$ is the ideal of \REF{bq}. 
We also put $q = q_{0}\pi^{2\ell}$ with 
$q_{0} \in \g^{\times} \cup \pi\g^{\times}$ and $0 \le \ell \in \mathbf{Z}$. 

\subsection{Case $t = 0$} 

In this case $K = F$ and $Q(\p) = M_{2}(F)$ by \REF{cd4}. 
We take a Witt decomposition 
\begin{gather}
V = (Fe_{1} + Ff_{1}) \oplus (Fe_{2} + Ff_{2}), \qquad 
L = (\g e_{1} + \g f_{1}) + (\g e_{2} + \g f_{2}) \label{t0} 
\end{gather}
with some elements $e_{i}$ and $f_{i}$ of $V$ such that 
$\p(e_{i},\, e_{j}) = \p(f_{i},\, f_{j}) = 0$ and 
$2\p_{v}(e_{i},\, f_{j}) = 1$ or $0$ according as $i = j$ or $i \ne j$. 
We put $q = \pi^{\n}\e$ with $\e \in \gx$. 
By \REF{char} and \REF{cd3}, 
$\psi$ is isotropic. 
We may identify $h$ with $q\pi^{-m}e_{1} + \pi^{m}f_{1}$ and 
$W$ with $(F(q\pi^{-m}e_{1} + \pi^{m}f_{1}))^{\perp}$ under some element of $C(L)$ 
by \REF{idrc}. 
Put $e = e_{2}$, $f = f_{2}$, and $k = \pi^{-\ell}(q\pi^{-m}e_{1} - \pi^{m}f_{1})$. 
Then $W = Fk \oplus (Fe + Ff)$ is a Witt decomposition, 
$M = \g k + \g e + \g f$ is a maximal lattice in $W$, 
and $L \cap W = \g \pi^{\la-m}k + \g e + \g f$. 
Since $\{k,\, e+f,\, e-f\}$ is an orthogonal basis of $W$ with respect to $\psi$, 
we set $\xi = k(e+f)(e-f)$. 
By \REF{is3}, 
$\xi$ gives an isomorphism of $(W,\, \psi)$ onto $(\awc,\, d^{-1}\nc)$ 
with $d = \e$ if $\n$ is even and $d = \pi\e$ if $\n$ is odd. 
We have $M\xi = \gt\zeta_{0}$ and $\am = \gf\zeta$ with 
\begin{gather*}
\zeta_{0} = 
\begin{bmatrix}
   		-d(1-2ef) \\ 
   		-ke \\ 
   		kf \\
\end{bmatrix},\qquad 
\zeta = 
\begin{bmatrix}
   		1 \\ 
   		ke \\ 
   		kf \\ 
   		ef \\
\end{bmatrix}. 
\end{gather*}
From these it follows that $(L \cap W)\xi = \gt\mathrm{diag}[\pi^{\la-m},\ 1,\ 1]\zeta_{0}$ 
and $\alw = \gf\mathrm{diag}[1,\ \pi^{\la-m},\ \pi^{\la-m},\ 1]\zeta$. 
The matrix $\n_{0}$ that represents the norm form $\n$ of $\aw$ with respect to $\zeta$ is given by 
\begin{gather*}
\n_{0} = 
\begin{bmatrix}
   		1 & 0 & 0 & 1/2 \\ 
   		0 & 0 & -d/2 & 0 \\ 
   		0 & -d/2 & 0 & 0 \\ 
   		1/2 & 0 & 0 & 0 \\ 
\end{bmatrix}. 
\end{gather*}
We have then $\alwd = \gf\eta$ with 
$\eta = \mathrm{diag}[1,\ \pi^{m-\la},\ \pi^{m-\la},\ 1](2\n_{0})^{-1}\zeta$. 
It can be seen that 
\begin{gather*}
\zeta_{0} = 
\begin{bmatrix}
   		-d & 0 & 0 & 2d \\ 
   		0 & -1 & 0 & 0 \\ 
   		0 & 0 & 1 & 0 \\ 
\end{bmatrix}\zeta. 
\end{gather*}
Hence we can put $\zeta_{0}=p\eta$ with $p\in\g_{4}^{3}$, 
which gives a $\g$-basis of $M\xi$ in $\awc$. 
For $x=x_{0}p\eta\in\awc$ with $x_{0}\in F_{3}^{1}$ 
observe that $x\in\alwd$ if and only if $x_{0}p\in\gf$. 
From this we see that 
$\alwd \cap \awc = \gt\mathrm{diag}[1,\ \pi^{m-\la},\ \pi^{m-\la}]\zeta_{0}$ if $\n$ is even and 
$\alwd \cap \awc = \gt\pi^{-1}\mathrm{diag}[1,\ \pi^{m-\la},\ \pi^{m-\la}]\zeta_{0}$ if $\n$ is odd. 
Comparing these with the expression of $(L \cap W)\xi$ above, 
we have $(L \cap W)\xi = \pi^{\la-m}(\alwd \cap \awc)$ if $\n$ is even 
and $(L \cap W)\xi = \pi^{\la+1-m}(\alwd \cap \awc)$ if $\n$ is odd. 
\\

\subsection{Case $t = 2$ and $(K/\pe) = -1$} 

In this case $K \ne F$ and we may assume $\dl \in \gx$. 
We take a Witt decomposition of $(V,\, \p)$ in \REF{w02} with 
$c = g^{2} \in \g^{\times}$ or $c \in \pi \g^{\times}$ according as $\pe \nmid D_{B}$ or $\pe \mid D_{B}$. 
Here $B = Q(\p)$. 
We put $q = \pi^{\n}\e$ with $\e \in \gx$. 
When $\pe \mid 2$, 
we put $2\g = \pe^{\ka}$ and $\dl = (1 + \pi^{2\ka}a)b^{-2}$ with $a,\, b \in \gx$ 
by \cite[Lemma 3.5(2)]{99b}. 
\\

Suppose $q \in \p[Kg]$. 
Then $\psi$ is isotropic. 
We may take $g$ so that $g^{2} = \e$ or $\pi\e$ according as $\n$ is even or odd. 
Also we may identify $W$ and a maximal lattice $M$ with those of \REF{t2-11} in either case. 
An orthogonal basis of $W$ is given by 
$W = F\sqd g \oplus F(e + f_{0}) \oplus F(e - f_{0})$, 
where $f_{0}$ is given in \REF{f0}. 
We set $\xi = \sqd g(e + f_{0})(e - f_{0})$; 
then $\xi\xi^{*} = \dl \e \in \gx$ if $\n$ is even and 
$\xi\xi^{*} = \pi\dl\e \in \pi\gx$ if $\n$ is odd. 
To obtain a $\g$-basis of $\alwd \cap \awc$, 
we are going to express $M\xi$ by means of a suitable basis of $\alwd$. 
Put $\zeta_{0} = \tp{[\sqd g\xi\ e\xi\ f_{0}\xi]} 
	= \tp{[-g^{2}\dl(1 - 2ef_{0})\ -\sqd ge\ \sqd gf_{0}]}$, 
which gives a $\g$-basis of $M\xi$. 
On the other hand, 
we have 
$\am = \gf \zeta$ and $\alw = \gf \alpha\zeta$, 
where $\zeta$ is given by \REF{amt2-1} for any $\n$; 
$\alpha$ is given by \REF{alw2-110} below if $\n$ is even 
and by \REF{alw2-111} or \REF{310} if $\n$ is odd. 
Then $\alwd = \gf \,\tp{\alpha}^{-1}(2\n_{0})^{-1}\zeta$, 
where $\n_{0}$ is the matrix representing the norm form $\n$ of $\aw$ with respect to $\zeta$; 
\begin{gather*}
\n_{0} = 
\begin{bmatrix}
   		1 & 0 & 0 & 1/2 \\ 
   		0 & 0 & -g^{2}\dl/2 & 0 \\ 
   		0 & -g^{2}\dl/2 & 0 & 0 \\ 
   		1/2 & 0 & 0 & 0 \\ 
\end{bmatrix}. 
\end{gather*}
Put $k = \tp{\alpha}^{-1}(2\n_{0})^{-1}\zeta$. 
We easily see that 
\begin{gather*}
\zeta_{0} = 
\begin{bmatrix}
   		-g^{2}\dl & 0 & 0 & 2g^{2}\dl \\ 
   		0 & -1 & 0 & 0 \\ 
   		0 & 0 & 1 & 0 \\ 
\end{bmatrix}\zeta. 
\end{gather*}

Let $\n$ be even. 
We first need to find $\alpha$ as mentioned. 
By applying the argument in \cite[{\S}3.2]{id} to the present case, 
we have $L \cap W = \gt\alpha_{0}\cdot\tp{[\sqd g\ e\ f_{0}]}$ 
with a matrix $\alpha_{0}$ given by 
\begin{gather*}
\alpha_{0} = 
\begin{bmatrix}
   		1 & 0 & 0 \\ 
   		0 & 1 & 0 \\ 
   		0 & (4q)^{-1}\pi^{\ell+m} & \pi^{\la-m} \\ 
\end{bmatrix} \quad \text{or} \quad 
\begin{bmatrix}
   		-b & 0 & 0 \\ 
   		0 & 1 & 0 \\ 
   		-(2q)^{-1}\pi^{\n}b & (4q)^{-1}\pi^{\ell+m} & \pi^{\la-m} \\ 
\end{bmatrix}
\end{gather*}
according as $\pe \nmid 2$ or not. 
Indeed, 
we can take $\e_{1} = 1$, $\la = \ell - m$, and $a_{1} = 0$ if $\pe \nmid 2$ and 
$\e_{1} = 1$, $\la = \ell - m$, and $a_{1} = \pi^{\ell}$ if $\pe \mid 2$ 
in the proof of Lemma 2.4(4) in \cite{id} with the notation there. 
Then a straightforward calculation gives $\alw = \gf\alpha\zeta$ with 
\begin{gather}
\alpha = 
\begin{bmatrix}
   		1 & 0 & 0 & 0 \\ 
   		0 & 1 & 0 & 0 \\ 
   		0 & (4q)^{-1}\pi^{\ell+m} & \pi^{\la-m} & 0 \\ 
   		0 & 0 & 0 & \pi^{\la-m} \\ 
\end{bmatrix} \quad \text{or} \nonumber \\ 
\begin{bmatrix}
   		1 & 0 & 0 & 0 \\ 
   		0 & 1 & 0 & 0 \\ 
   		(2q)^{-1}\pi^{\n}\dl\e b & (4q)^{-1}\pi^{\ell+m} & \pi^{\la-m} & 0 \\ 
   		0 & 0 & (2q)^{-1}\pi^{\n}b & \pi^{\la-m} \\ 
\end{bmatrix} 
\label{alw2-110}
\end{gather}
according as $\pe \nmid 2$ or not. 
Now, 
an expression of $M\xi$ that we require is given by 
\begin{eqnarray*}
M\xi &=& \gt 
\begin{bmatrix}
   		0 & 0 & 0 & -\pi^{\la-m}\dl\e \\ 
   		0 & 0 & \pi^{\la-m}\dl\e & 0 \\ 
   		0 & -\dl\e & -(4q)^{-1}\pi^{\ell+m}\dl\e & 0 \\
\end{bmatrix}k \quad \text{if $\pe \nmid 2$}, \\
M\xi &=& \gt 
\begin{bmatrix}
   		0 & 0 & 0 & -\pi^{\la-m}\dl\e \\ 
   		0 & 0 & \pi^{\la-m}\dl\e & 0 \\ 
   		0 & -\dl\e & -(4q)^{-1}\pi^{\ell+m}\dl\e & -(2q)^{-1}\pi^{\n}\dl\e b \\
\end{bmatrix}k \quad \text{if $\pe \mid 2$}. 
\end{eqnarray*}
Employing this expression, 
we obtain $\alwd \cap \awc = \gt\alpha_{1}\zeta_{0}$ with 
\begin{gather*}
\alpha_{1} = 
\begin{bmatrix}
   		\pi^{m-\la} & 0 & 0 \\ 
   		0 & \pi^{m-\la} & 0 \\ 
   		0 & (4q)^{-1}\pi^{2m} & 1 \\ 
\end{bmatrix} \quad \text{or} \quad 
\begin{bmatrix}
   		\pi^{m-\la} & 0 & 0 \\ 
   		0 & \pi^{m-\la} & 0 \\ 
   		-(2q)^{-1}\pi^{\ell+m}b & (4q)^{-1}\pi^{2m} & 1 \\ 
\end{bmatrix}
\end{gather*}
according as $\pe \nmid 2$ or not. 
On the other hand, 
$(L \cap W)\xi$ can be given by $\gt \alpha_{0}\zeta_{0}$. 
Comparing $\alpha_{1}$ with $\alpha_{0}$, 
we then find $(L \cap W)\xi = \pi^{\la-m}(\alwd \cap \awc)$. 

If $\n$ is odd, 
in the same manner as in the case of even $\n$, 
we have an expression of $M\xi$ with $k$. 
Employing that expression, 
we obtain $(L \cap W)\xi = \pi^{\la+1-m}(\alwd \cap \awc)$. 
\\

Suppose $q \not\in \p[Kg]$. 
Then $\psi$ is anisotropic. 
We may identify $W$ and $L \cap W$ with those of \REF{wl0} with $g^{2} = c \in \gx \cup \pi\gx$. 
In the present case, 
$c \in \pi\gx$ if and only if $\n$ is even. 
Put $u = \sqd$ or $2^{-1}(1 + b\sqd)$ according as $\pe \nmid 2$ or not; 
then $\mr = \g[u]$. 
We have $W = Fg \oplus F\sqd g \oplus Fk$, 
where $k = q\pi^{-m}e - \pi^{m}f$. 
We set $\xi = \pi^{-\ell-1}g\sqd gk$ if $\n$ is even and 
$\xi = \pi^{-\ell}g\sqd gk$ if $\n$ is odd. 
Since $M = \mr g + \g \pi^{-\ell}k$, 
we have $\am = \gf\zeta$ and $\alw = \gf\alpha\zeta$ with 
$\zeta = \tp{[1\ gug\ \pi^{-\ell}gk\ \pi^{-\ell}ugk]}$ and 
$\alpha = \mathrm{diag}[1,\ 1,\ \pi^{\la-m},\ \pi^{\la-m}]$. 

If $\pe \nmid 2$, 
then $\{g,\, \sqd g,\, \pi^{-m}k\}$ is an orthogonal $\g$-basis of $L \cap W$ with respect to $\psi$. 
By straightforward calculations it can be verified that 
$(L \cap W)\xi = \pi^{\la+1-m}(\alwd \cap \awc)$. 

If $\pe \mid 2$, 
in a similar way to the case of $q \in \p[Kg]$ and even $\n$, 
we see that $(L \cap W)\xi = \pi^{\la+1-m}(\alwd \cap \awc)$. 

\subsection{Case $t = 2$ and $(K/\pe) = 0$} 

We take a Witt decomposition of $(V,\, \p)$ in \REF{w02}. 
Put $2\g = \pe^{\ka}$. 
We may assume that $\dl \in \gx \cup \pi\gx$. 
\\

\textbf{Case $\dl \in \gx$.} 
Then $\pe \mid 2$. 
We employ the notation in the case where 
$t = 2$, $(K/\pe) = 0$, and $\dl \in \gx$ in the proof of \LEM{lalw}. 

Suppose $q \in \p[Kg]$. 
Then $\psi$ is isotropic. 
Our desired fact can be seen in a similar way to the case where $(K/\pe) = -1$ and $q \in \p[Kg]$. 
That is given as follows; 
$(L \cap W)\xi = \pi^{\la-m}(\alwd \cap \awc)$ if $\n$ is even 
and $(L \cap W)\xi = \pi^{\la+1-m}(\alwd \cap \awc)$ if $\n$ is odd 
with some $\xi$ such that $\xi\xi^{*} \in \gx$ if $\n$ is even and 
$\xi\xi^{*} \in \pi\gx$ if $\n$ is odd. 

Suppose $q \not\in \p[Kg]$. 
Then $\psi$ is anisotropic, 
and so $A^{+}(W)$ is a division quaternion algebra over $F$. 
The pair $(A^{+}(W),\, \n)$ can be viewed as a quadratic space of dimension $4$ over $F$. 
To prove our theorem in this case, 
we start with an auxiliary lemma as follows. 
\\

Let $D$ be a division quaternion algebra over a local field $F$. 
For an unramified quadratic field $J$ of $F$ contained in $D$, 
let $\om$ be an element of $D^{\times}$ such that 
$\om^{2} \in \pi\gx$ and $a\om = \om a^{\iota}$ for every $a \in J$. 
Here $\iota$ is the main involution of $D$. 
It can be seen by \cite[Theorem 5.13]{04} that 
\begin{gather}
D = J \oplus J\omega, \qquad \mo = \mr_{J} + \mr_{J}\om \label{unst}
\end{gather}
with respect to the norm form $\n$ of $D$, 
where $\mo$ is a unique maximal order in $D$ and $\mr_{J}$ is the maximal order of $J$. 
Let $\mr_{J} = \g[u]$ and put $v = u - u^{\iota}$; 
then $v \in \mr_{J}^{\times}$ and $v + v^{\iota} = 0$. 

\begin{lem} \label{rcfr}
Let $D = J + J\om$ be a division quaternion algebra over $F$ with the above notation. 
Put $\zeta = \tp{[1\ u\ \om\ u\om]}$. 
Let $(D^{\circ},\, d^{-1}\n^{\circ})$ be a quadratic space 
of dimension $3$ over $F$ with $d \in \gx \cup \pi\gx$ 
and $M$ a $\g$-maximal lattice in $D^{\circ}$ with respect to $d^{-1}\n^{\circ}$. 
Then the following assertions hold: 
\begin{enumerate}
\item $D^{\circ} = Fv \oplus J\om$ and $M = \g pv + \mr_{J}\om$, 
where $p = 1$ or $\pi$ according as $d \in \gx$ or $d\in\pi\gx$. 
\item Let $\mN$ be a $\g$-lattice in $D$. 
Put $\mN = \g_{4}^{1}\eta$, 
$\eta = P\zeta$ with $P \in GL_{4}(F)$, 
and 
\begin{gather*}
A = 
\begin{bmatrix}
   		-\pi^{\ell}(u + u^{\iota}) & 2\pi^{\ell} & 0 & 0 \\ 
   		0 & 0 & 1 & 0 \\ 
   		0 & 0 & 0 & 1 \\
\end{bmatrix}
P^{-1}, 
\end{gather*}
where $\ell = 0$ or $1$ according as $d \in \gx$ or $d\in\pi\gx$. 
Then there are $Q_{1} \in GL_{3}(\g)$ and $Q_{2} \in GL_{4}(\g)$ such that 
\begin{gather*}
Q_{1}AQ_{2} = 
\begin{bmatrix}
   		\e_{1} & 0 & 0 & 0 \\ 
   		0 & \e_{2} & 0 & 0 \\ 
   		0 & 0 & \e_{3} & 0 \\
\end{bmatrix}
\end{gather*}
with the elementary divisors $\{\e_{1}\g,\, \e_{2}\g,\, \e_{3}\g\}$ of $A$. 
\item In the setting of (2), 
$\mN \cap D^{\circ} = \g_{3}^{1} \mathrm{diag}[\e_{1}^{-1},\, \e_{2}^{-1},\, \e_{3}^{-1}]Q_{1}\zeta_{0}$ 
and $[M/\mN \cap D^{\circ}] = (\e_{1}\e_{2}\e_{3})^{-1}\g$, 
where $\zeta_{0} = \tp{[pv\ \om\ u\om]}$ is the basis of $M$ given in (1). 
\end{enumerate}
\end{lem}
\begin{proof}
(1) Clearly, 
$(D^{\circ},\, d^{-1}\n^{\circ})$ is anisotropic. 
Since all $v$, $\om$, and $v\om$ belong to $D^{\circ}$, 
we have $D^{\circ} = Fv \oplus J\om$. 
If $d \in \gx$, 
the second assertion is found in \cite[{\S}7.7\ (III)]{04}. 
If $d \in \pi\gx$, 
we put $M^{\prime} = \g \pi v + \mr_{J}\om$, 
which is an integral lattice in $(D^{\circ},\, d^{-1}\n^{\circ})$. 
Observe $(\mr_{J}\om)\widetilde{\,} = d\pi^{-1}\mr_{J}\om$. 
Then straightforward calculations show that 
$[\widetilde{M^{\prime}}/M^{\prime}] = 
	[(\g \pi v)\widetilde{\,}/\g \pi v][(\mr_{J}\om)\widetilde{\,}/\mr_{J}\om] = 2\pe$. 
This coincides with the discriminant ideal of $(D^{\circ},\, d^{-1}\n^{\circ})$, 
and so $M^{\prime}$ is the maximal lattice $M$. 

(2) Since the matrix $A$ has rank $3$, 
our assertion is an application of the theory of elementary divisors. 

(3) Put $\zeta_{0} = \tp{[pv\ \om\ u\om]}$; 
then $\zeta_{0} = A\eta$ as $v = -(u + u^{\iota}) + 2u$. 
We set $\tp{[v_{1}\ v_{2}\ v_{3}]} = Q_{1}\zeta_{0}$ 
and $\tp{[k_{1}\ k_{2}\ k_{3}\ k_{4}]} = Q_{2}^{-1}\eta$, 
which give $\g$-bases of $M$ and $\mN$, respectively. 
Then $v_{i} = \e_{i}k_{i}$ for $i = 1,\, 2,\, 3$ by (2). 
Thus we have $\mN \cap D^{\circ} = \sum_{i=1}^{3} \g \e_{i}^{-1}v_{i}$. 
From this we see that 
$[M/\mN \cap D^{\circ}] = 
	[\g_{3}^{1}/\g_{3}^{1}\mathrm{diag}[\e_{1}^{-1},\, \e_{2}^{-1},\, \e_{3}^{-1}]] = 
	(\e_{1}\e_{2}\e_{3})^{-1}\g$, 
which proves (3). 
\end{proof}
\vspace{5mm}

Returning to the present case, 
we take our setting to be that of the case where $q \not\in \p[Kg]$ in {\S}2.3 
with $q = (-\pi a)^{\n}\e$ and $g^{2} = s\e \in \gx$. 
The anisotropic space $(W,\, \psi)$ and maximal lattice $M$ 
may be identified with those given in \REF{t2003} and \REF{m20030} or \REF{320}, 
respectively; see \REF{idrc} and its explanation. 
Let $\xi$ be an isomorphism in \REF{is3} so that $\xi\xi^{*} \in \gx \cup \pi\gx$. 
Then we have a maximal lattice $M\xi$ in $(A^{+}(W)^{\circ},\, (\xi\xi^{*})^{-1}\n^{\circ})$. 
Put $s = 1 + \pi^{2(\ka-k)-1}s_{0} \in \gx$ with $s_{0} \in \gx$. 

Our aim is to determine $[M\xi/\alwd \cap \awc]$ 
by applying \LEM{rcfr} to $D = \aw$, $M = M\xi$, and $\mN = \alwd$. 
To do this, there are the following steps: 
\begin{enumerate}
\item Find a structure of $\aw$ as in \REF{unst} with $J$ and $\om$. 
Then $\zeta = \tp{[1\ u\ \om\ u\om]}$ gives an $F$-basis of $\aw$ 
and $\tp{[pv\ \om\ u\om]}$ gives a $\g$-basis of $M\xi$ 
with $v = u - u^{*}$ and $p$ as in \LEM{rcfr}(1). 
\item Compute the matrix $\n_{0} \in GL_{4}(F)$ that represents $\n$ with respect to $\zeta$. 
\item Express $\am = \gf \zeta_{1}$ and $\alw = \gf \alpha\zeta_{1}$ 
with suitable $\zeta_{1}$ and $\alpha \in GL_{4}(F)$. 
\item Compute the matrix $P \in GL_{4}(F) \cap M_{4}(\g)$ such that $\zeta_{1} = P\zeta$; 
then $\alwd = \gf \,\tp{(\alpha P)}^{-1}(2\n_{0})^{-1}\zeta$. 
Put $\eta = \tp{(\alpha P)}^{-1}(2\n_{0})^{-1}\zeta$. 
\item Let $A$ be the matrix in $F^{3}_{4}$ determined by $M\xi = \gt A\eta$; 
then \LEM{rcfr}(2) is applicable to this $A$. 
Find the elementary divisors $\{\e_{1}\g,\ \e_{2}\g,$\ $\e_{3}\g\}$ of $A$. 
\item $[M\xi/\alwd \cap \awc] = (\e_{1}\e_{2}\e_{3})^{-1}\g$ by \LEM{rcfr}(3). 
\end{enumerate} 
\vspace{5mm}

Let $\n$ be even. 
Then $W = F\sqd g \oplus (Fg_{2} + Fg_{3})$ and 
$M = \g \sqd g + \g \pi^{m-\la}g_{2} + \g (\pi^{-k}\sqd g + cg_{3})$, 
where we put $u_{1} = \pi^{-k}(1 + b\sqd)$, 
$z = (-\pi a)^{\ell}$, 
$z_{0} = (\pi a)^{\ell}$, 
and $c = b^{-1}\pi^{-k-m}2z_{0}s\e$; 
$g_{2},\, g_{3}$ are given in \REF{g2g3}. 
It is remarked that 
we use the letter $u_{1}$ in the present setting 
instead of $u$ there. 
We have then 
\begin{gather}
W = F\sqd g \oplus Fg_{2} \oplus Fg_{3}^{\prime}, \qquad 
g_{3}^{\prime} = \frac{\pi^{2m}}{q(1-s)}g_{2} + 2g_{3}. \label{wob20030}
\end{gather}
We set $\xi = \sqd gg_{2}g_{3}^{\prime}$. 
This $\xi$ defines an isomorphism in \REF{is3} so that 
$\xi\xi^{*} = \dl \e \in \gx$. 
We put 
\begin{eqnarray*}
v_{0} &=& \frac{1}{\dl \e}\sqd g + \frac{\pi^{m}b}{z_{0}\e}g_{2} \in W, \\
g^{\prime} &=& \frac{1}{\pi^{k}}\sqd g + 
	\frac{z_{0}s\e}{\pi^{k+m}b}\cdot\left\{\frac{\pi^{2m}}{q(s-1)}g_{2} + g_{3}^{\prime}\right\} \in M. 
\end{eqnarray*}
We also put $v = v_{0}\xi \in \awc$ and $\om = g^{\prime}\xi \in M\xi \subset \awc$. 
It can be seen that 
$\p(v_{0},\, g^{\prime}) = 0$, 
$\n[v] = \xi\xi^{*}\p[v_{0}] = \pi^{2\kappa}s_{0}a - 1$, 
and $\om^{2} = -\n[\om] = b^{-2}\pi\dl s\e^{2}a \in \pi\gx$. 
Then $J = F + Fv$ becomes a quadratic field of $F$ and $v^{2} \in 1 + \pi^{2\ka}\gx$. 
By \cite[Lemma 3.5(2)]{99b}, 
$J$ is an unramified quadratic extension over $F$. 
Further $x\om = \om x^{*}$ for every $x \in J$ by $\p(v_{0},\, g^{\prime}) = 0$. 
Thus we have $\aw = J \oplus J\om$ with respect to $\n$ by \REF{unst}. 
The order $\am$ has discriminant $\pe^{2}$ and so 
it is contained in the maximal order $\mr_{J} + \mr_{J}\om$ in $\aw$. 
Here $\mr_{J} = \g[u]$ is the maximal order of $J$ with $u = 2^{-1}(1 + v)$. 
Put $\zeta = \tp{[1\ u\ \om\ u\om]}$. 
Because $M\xi$ is a maximal lattice in $(\awc,\, (\xi\xi^{*})^{-1}\nc)$, 
$M\xi = \gt \,\tp{[v\ \om\ u\om]}$ by \LEM{rcfr}(1). 
The matrix $\n_{0}$ representing $\n$ with respect to $\zeta$ is given by 
\begin{gather}
\n_{0} = 
\begin{bmatrix}
   		1 & 1/2 & 0 & 0 \\ 
   		1/2 & \n[u] & 0 & 0 \\ 
   		0 & 0 & \n[\om] & \n[\om]/2 \\
   		0 & 0 & \n[\om]/2 & \n[u\om] \\
\end{bmatrix}, \qquad 
\n[u] = \frac{1-v^{2}}{4}. \label{nrep}
\end{gather}

Now we have $\am = \gf \zeta_{1}$ and $\alw = \gf \alpha\zeta_{1}$, 
where $\zeta_{1}$ is given by \REF{am20030} and $\alpha$ by \REF{alw20030}. 
These $v$, $\zeta$, and $\zeta_{1}$ are expressed with $1,\ \sqd g,\ g_{2}$, and $g_{3}^{\prime}$ as follows: 
\begin{eqnarray*}
v &=& -\frac{z_{0}(s-1)b}{\pi^{m}}\sqd gg_{3}^{\prime} - sg_{2}g_{3}^{\prime}, \\ 
\zeta &=& 
\begin{bmatrix}
	1 \\
	\frac{1}{2} - \frac{z_{0}(s-1)b}{2\pi^{m}}\sqd gg_{3}^{\prime} - \frac{s}{2}g_{2}g_{3}^{\prime} \\
	-\frac{\pi^{m}}{\pi^{k}z_{0}(s-1)b}\sqd gg_{2} - 
	\frac{z_{0}s\e}{\pi^{k+m}b}\sqd gg_{3}^{\prime} - \frac{\dl s\e}{\pi^{k}}g_{2}g_{3}^{\prime} \\
	-\frac{\pi^{m}(1+v^{2})}{2\pi^{k}z_{0}(s-1)b}\sqd g\gb - 
	\frac{z_{0}s\e}{\pi^{k+m}b}\sqd g\gbp - \frac{\dl s\e}{\pi^{k}}\gb \gbp \\
\end{bmatrix}, \\
\zeta_{1} &=& 
\begin{bmatrix}
	1 \\
	\pi^{m-\la}\sqd gg_{2} \\
	-\frac{\dl s\e}{\pi^{k}} + \frac{\pi^{m}s}{\pi^{k}z_{0}(s-1)b}\sqd gg_{2} 
		+ \frac{z_{0}s\e}{\pi^{k+m}b}\sqd gg_{3}^{\prime} \\ 
	\frac{\pi z_{0}s\e}{\pi^{\ell+\ka}b} - 
	\frac{\pi^{m+1}}{\pi^{\ell+\ka}}\sqd g \gb + 
	\frac{\pi z_{0}s\e}{\pi^{\ell+\ka}b}\gb \gbp \\
\end{bmatrix}. 
\end{eqnarray*}
Let $\zeta_{1} = P\zeta$ with $P \in GL_{4}(F) \cap M_{4}(\g)$. 
By straightforward computations $P$ is given by  
\begin{gather*}
P = 
\begin{bmatrix}
   		1 & 0 & 0 & 0 \\ 
   		0 & 0 & -\frac{2z_{0}b}{\pi^{\ell+\ka}a} & \frac{2z_{0}b}{\pi^{\ell+\ka}a} \\ 
   		-\frac{\dl s\e(1+v^{2})}{\pi^{k}v^{2}} & \frac{2\dl s\e}{\pi^{k}v^{2}} 
			& \frac{s}{v^{2}} & -\frac{2s}{v^{2}} \\ 
   		\frac{\pi z_{0}s\e(1+v^{2})}{\pi^{\ell+\ka}bv^{2}} & 
			-\frac{2\pi z_{0}s\e}{\pi^{\ell+\ka}bv^{2}} & 
			-\frac{\pi^{\ka}z_{0}bs_{0}}{\pi^{\ell+k}v^{2}} & 
			\frac{2\pi^{\ka}z_{0}bs_{0}}{\pi^{\ell+k}v^{2}} \\ 
\end{bmatrix}. 
\end{gather*}
Then $\alwd = \gf \,\tp{(\alpha P)}^{-1}(2\n_{0})^{-1}\zeta$. 
Putting $\eta = \tp{(\alpha P)}^{-1}(2\n_{0})^{-1}\zeta$, 
we have $M\xi = \gt A\eta$ with 
\begin{gather*}
A = 
\begin{bmatrix}
   		-1 & 2 & 0 & 0 \\ 
   		0 & 0 & 1 & 0 \\ 
   		0 & 0 & 0 & 1 \\
\end{bmatrix} 2\n_{0}\cdot\tp{(\alpha P)} 
=
\begin{bmatrix}
   		0 & 0 & -\frac{2\dl s\e}{\pi^{k}} & \frac{2z_{0}s\e}{\pi^{k+m}b} \\ 
   		0 & \frac{2z_{0}\om^{2}b}{\pi^{k+m+1}a} & 0 & 0 \\ 
   		0 & \frac{2z_{0}\om^{2}b(1-2\n[u])}{\pi^{k+m+1}a} 
		& -\om^{2}s & \frac{z_{0}\om^{2}(s-1)b}{\pi^{m}} \\
\end{bmatrix}. 
\end{gather*}
We are going to apply \LEM{rcfr}(2) to $\alwd$ and $M\xi$. 
Put 
\begin{gather*}
Q_{1} = 
\begin{bmatrix}
   		1 & \frac{2\dl \e(1-2\n[u])}{\pi^{k}\om^{2}} & -\frac{2\dl \e}{\pi^{k}\om^{2}} \\ 
   		0 & 1 & 0 \\ 
   		0 & 2\n[u]-1 & 1 \\
\end{bmatrix}, \qquad 
Q_{2} = 
\begin{bmatrix}
   		1 & 0 & 0 & 0 \\ 
   		0 & 1 & 0 & 0 \\ 
   		0 & 0 & 1 & \frac{z_{0}(s-1)b}{\pi^{m}s}  \\
   		0 & 0 & 0 & 1 \\ 
\end{bmatrix}. 
\end{gather*}
Then we find that 
\begin{gather*}
Q_{1}AQ_{2} = 
\begin{bmatrix}
   		0 & 0 & 0 & \frac{2z_{0}\e v^{2}}{\pi^{k+m}b} \\ 
   		0 & \frac{2z_{0}\om^{2}b}{\pi^{k+m+1}a} & 0 & 0 \\
   		0 & 0 & -\om^{2}s & 0 \\ 
\end{bmatrix}, 
\end{gather*}
which gives the elementary divisors of $A$. 
Thus \LEM{rcfr}(3) shows that 
\begin{gather*}
\alwd \cap \awc = \g \pi^{m-\la-1}v_{1} + \g \pi^{m-\la-1}v_{2} + \g \pi^{-1}v_{3}, \\ 
[M\xi/\alwd \cap \awc] = \pi^{-2(\la-m)-3}\g = \pe^{-3}[M/L \cap W]^{-2}, 
\end{gather*}
where $\tp{[v_{1}\ v_{2}\ v_{3}]} = Q_{1}\cdot\tp{[v\ \om\ u\om]}$. 

We next focus on the two lattices $(L \cap W)\xi$ and an integral lattice $N$ in $\awc$ defined by 
\begin{gather*}
N = \pi^{\la+1-m}(\alwd \cap \awc) = \gt 
\begin{bmatrix}
   		1 & 0 & 0 \\ 
   		0 & 1 & 0 \\ 
   		0 & 0 & \pi^{\la-m} \\
\end{bmatrix}
Q_{1}
\begin{bmatrix}
   		v \\ 
   		\om \\ 
   		u\om \\ 
\end{bmatrix}. 
\end{gather*} 
By \cite[(3.8)]{id}, 
$L \cap W$ can be obtained from $M = \gt \zeta_{0}$ of \REF{m20030} as follows:  
\begin{gather*}
L \cap W = \gt \alpha_{0}\zeta_{0} = \gt 
\begin{bmatrix}
   		1 & 0 & 0 \\ 
   		0 & \pi^{\la-m} & 0 \\ 
   		c_{0} & 0 & 1 \\
\end{bmatrix} 
\begin{bmatrix}
   		\sqd g \\ 
   		\pi^{m-\la}g_{2} \\ 
   		\pi^{-k}\sqd g + cg_{3} \\
\end{bmatrix} 
\end{gather*} 
with $c_{0} = -\pi^{-k}(1 + (-1)^{\ell}) \in \g$. 
Observing 
\begin{gather*}
M\xi = \gt \zeta_{0}\xi = \gt 
\begin{bmatrix}
   		-\dl s\e \gb \gbp \\ 
   		-\pi^{-\la-m}q(s-1)\sqd g\gbp \\ 
   		\om \\
\end{bmatrix}, 
\end{gather*}
we have $(L \cap W)\xi = \gt \alpha_{0}P_{0}\cdot\tp{[v\ \om\ u\om]}$ 
with the matrix $P_{0} \in GL_{3}(\g)$ given by 
\begin{gather*}
P_{0} = 
\begin{bmatrix}
   		\frac{\dl s\e}{v^{2}} & \frac{b^{2}\dl(1+v^{2})}{\pi^{k+1}v^{2}a} 
		& -\frac{2b^{2}\dl}{\pi^{k+1}v^{2}a} \\ 
   		-\frac{\pi^{\ka}z_{0}b\dl \e s_{0}}{\pi^{\ell+k}v^{2}} & 
		-\frac{z_{0}b(1+v^{2})}{\pi^{\ell+\ka}v^{2}a} &  \frac{2z_{0}b}{\pi^{\ell+\ka}v^{2}a} \\ 
   		0 & 1 & 0 \\
\end{bmatrix}. 
\end{gather*} 
Now $(L \cap W)\xi \subset N$ if and only if 
$\alpha_{0}P_{0}Q_{1}^{-1}\mathrm{diag}[1\ 1\ \pi^{m-\la}] \in M_{3}(\g)$. 
The latter matrix is of the form 
\begin{gather*}
\begin{bmatrix}
   		* & * & x \\ 
   		* & * & -\frac{\pi^{\ka}z_{0}b\dl s_{0}}{\pi^{2k+1}v^{2}}\cdot
		\frac{2b^{2}}{\pi^{\ell}sa}+\frac{2z_{0}b}{\pi^{\ell+\ka}v^{2}a} \\ 
   		* & * & c_{0}x \\
\end{bmatrix}, \qquad 
x = \frac{\dl s\e}{v^{2}}\cdot
		\frac{2\pi^{m}b^{2}}{\pi^{\ell+\ka}s\e a}-\frac{2\pi^{m}b^{2}\dl}{\pi^{\ell+\ka}v^{2}a} = 0.  
\end{gather*} 
Here all entries in $*$ belong to $\g$. 
Hence it follows that $(L \cap W)\xi \subset N$. 
Since $[M\xi/N] = \pe^{\la-m} = [M/L \cap W]$, 
$(L \cap W)\xi$ must coincide with $\pi^{\la+1-m}(\alwd \cap \awc)$. 
\\

Let $\n$ be odd. 
Then $W = F(b\dl + \sqd)g \oplus (Fg_{2} + Fg_{3})$ and 
$M = \g g_{1}g + \g \pi^{m-\la}g_{2} + \g (\pi^{-k-1}g_{1}g + cg_{3})$, 
where we put $u_{1} = \pi^{-k}(1 + b\sqd)$, 
$z = (-\pi a)^{\ell}u_{1}$, 
$z_{0} = (\pi a)^{\ell+1}$, 
$g_{1} = \pi^{-k}(b\dl + \sqd) \in \mr$, 
and $c = b^{-1}\pi^{-k-m-1}2z_{0}s\e$; 
$g_{2},\, g_{3}$ are given in \REF{g2g3} 
(see\ also\ \REF{320}). 
Here we use the letter $u_{1}$ in the present setting instead of $u$ there. 
We have then $W = Fg_{1}g \oplus Fg_{2} \oplus Fg_{3}^{\prime}$ 
with $g_{3}^{\prime} = (q(1-s))^{-1}\pi^{2m}g_{2} + 2g_{3}$. 
We set $\xi = g_{1}gg_{2}g_{3}^{\prime}$. 
This $\xi$ defines an isomorphism in \REF{is3} so that 
$\xi\xi^{*} = -\pi\dl \e a \in \pi\gx$. 
Then we find that $(L \cap W)\xi = \pi^{\la+1-m}(\alwd \cap \awc)$. 
This can be proven in the same way as in the case of even $\n$. 
\\

\textbf{Case $\dl \in \pi\gx$.} 
Our proof is similar to the previous case $\dl \in \gx$. 
When $q \in \p[Kg]$, 
it can be seen that $(L \cap W)\xi = \pi^{\la+1-m}(\alwd \cap \awc)$ if $\n$ is even 
and $(L \cap W)\xi = \pi^{\la-m}(\alwd \cap \awc)$ if $\n$ is odd. 
Here $\xi$ is a specified isomorphism in \REF{is3} so that $\xi\xi^{*} \in \pi\gx$ if $\n$ is even 
and $\xi\xi^{*} \in \gx$ if $\n$ is odd. 
When $q \not\in \p[Kg]$, 
by following steps (1) - (6) explained after \LEM{rcfr}, 
we can verify that $(L \cap W)\xi = \pi^{\la+1-m}(\alwd \cap \awc)$ if $\pe \mid 2$ 
with some $\xi$ such that $\xi\xi^{*} \in \pi\gx$ if $\n$ is even 
and $\xi\xi^{*} \in \gx$ if $\n$ is odd. 
Also if $\pe \nmid 2$, 
we have $(L \cap W)\xi = \pi^{\la+1-m}(\alwd \cap \awc)$ in a straightforward way, 
where $\xi$ satisfies $\xi\xi^{*} \in \pi\gx$ if $\n$ is even 
and $\xi\xi^{*} \in \gx$ if $\n$ is odd. 
This settles the case $t = 2$. 

\subsection{Case $t = 4$} 

Then $K = F$ and $Q(\p)$ is a division algebra. 
Because $L$ is a unique maximal lattice in the anisotropic space $(V,\, \p)$, 
so is $L \cap W$ in $W$ with respect to anisotropic $\psi$. 
Hence, 
by \THM{shi}(2), 
$L \cap W$ can be given by 
\begin{gather}
(L \cap W)\xi = 
\begin{cases}
S^{+}_{W} \cap \awc & \text{if $\n$ is even}, \\
\pi(\alwd \cap \awc) & \text{if $\n$ is odd}. 
\end{cases} \label{rec4}
\end{gather}
Here $\xi$ is an isomorphism of \REF{is3} so that 
$\xi\xi^{*} \in \gx$ or $\xi\xi^{*} \in \pi\gx$ according as $\n$ is even or odd, 
and $S^{+}_{W}$ is a unique maximal order in the division algebra $\aw$. 
Clearly \REF{rec4} proves the case of odd $\n$. 
Hereafter in the case $t = 4$ we assume that $\n$ is even. 
Put $B = Q(\p)$, $K_{1} = F(\sq)$, and $q = \pi^{\n}\e$ with $\e \in \gx$; 
also put $2\g = \pe^{\ka}$. 

Assume that $K_{1} \ne F$. 
We employ the setting in the case of {\S}2.4 where 
$t = 4$, $q_{0}( = \e) \in \gx$, and $K_{1} \ne F$. 
Then $(V,\, \p) = (B,\, \beta)$ and $B = K_{1} + K_{1}\om$ 
with $\om^{2} = c \in \gx \cup \pi\gx$ given there. 
In view of \REF{idrc} we may identify $W$ with that of \REF{w4}. 
Let $1_{B}$ be the identity element of $B$ and put $v = \sqrt{-\e}$. 

Suppose $(K_{1}/\pe) = 0$. 
Then $\pe \mid 2$. 
We use the same notation as in the case $(K_{1}/\pe) = 0$ in {\S}2.4. 
Since $W = F1_{B} \oplus F\om \oplus Fv\om$, 
we set $\xi = 1_{B}\om(v\om)$. 
Then $\xi$ gives an isomorphism as in \REF{is3} and $\xi\xi^{*} = \e c^{2} \in \gx$. 
By \REF{lw40}, 
$(L \cap W)\xi$ is given by 
\begin{gather*}
(L \cap W)\xi = \gt\zeta_{0} = \gt 
\begin{bmatrix}
   		\om(v\om) \\ 
   		\pi^{-k}c(\ib(v\om) - \e b\ib\om) \\ 
   		\pi^{k+1-\ka}(\om(v\om) + c\ib(v\om)) \\
\end{bmatrix}. 
\end{gather*}
On the other hand, 
by \REF{alw40}, 
$\alw$ is expressed as follows: 
\begin{gather*}
A^{+}(L \cap W) = \gf\zeta = \gf 
\begin{bmatrix}
   		1 \\ 
   		\pi^{-k}(\ib\om + b\ib(v\om)) \\ 
   		\pi^{k+1-\ka}(1 + \ib\om) \\ 
   		-\pi^{1-\ka}(c + \ib\om + b\ib(v\om) + b\om(v\om)) \\
\end{bmatrix}. 
\end{gather*}
Then $\alwd = \gf(2\n_{0})^{-1}\zeta$. 
Here $\n_{0}$ is the matrix that represents the norm form $\n$ of $\aw$ with respect to $\zeta$, 
which is given by 
\begin{gather*}
\n_{0} = 
\begin{bmatrix}
   		1 & 0 & \pi^{k+1-\ka} & -\pi^{1-\ka}c \\ 
   		0 & \pi ac & -\pi^{1-\ka}c & -\pi^{k+2-\ka}ac \\ 
  		\pi^{k+1-\ka} & -\pi^{1-\ka}c & \pi^{2-2(\ka-k)}(1-c) & 0 \\
   		-\pi^{1-\ka}c & -\pi^{k+2-\ka}ac & 0 & \pi^{3-2(\ka-k)}ac(1-c) \\ 
\end{bmatrix}. 
\end{gather*}
By straightforward computations we have 
\begin{gather*}
\zeta_{0} = -\frac{1}{b} 
\begin{bmatrix}
   		c & \pi^{k} & 0 & \pi^{\ka-1} \\ 
   		\pi^{k+1}ac & -c & -\pi^{\ka}ac & 0 \\ 
   		0 & \pi^{2k+1-\ka}(1-c) & c & \pi^{k} \\
\end{bmatrix}\zeta. 
\end{gather*}
From these $(L \cap W)\xi$ can be expressed by means of $\eta = (2\n_{0})^{-1}\zeta$ as follows: 
\begin{gather*}
(L \cap W)\xi = \gt 
\begin{bmatrix}
   		0 & 0 & 0 & -2\pi^{1-\ka}\e bc^{2} \\ 
   		0 & 0 & 2\pi^{1-\ka}\e bc^{2} & 0 \\ 
   		0 & -2\pi^{1-\ka}\e bc^{2} & 0 & 0 \\
\end{bmatrix}\eta. 
\end{gather*}
Employing this, 
we find that $\alwd \cap \awc = \pi^{-1}\gt\zeta_{0}$, 
and therefore $(L \cap W)\xi = \pi(\alwd \cap \awc)$. 

Similarly for $(K_{1}/\pe) = -1$, 
we see that 
$(L \cap W)\xi = \pi(\alwd \cap \awc)$ with some $\xi$ satisfying $\xi\xi^{*} \in \gx$.

Assume that $K_{1} = F$. 
We employ the setting in the case where $q_{0} \in \gx$ and $K_{1} = F$ in {\S}2.4. 
Then $(V,\, \p) = (B,\, -\beta)$, 
$B = J + J\om$ with $\om^{2} = \pi$, 
and $J = F(\sqrt{s})$ is an unramified quadratic field of $F$ 
with $s \in \gx$ given there. 
We may identify $W$ with $F\sqrt{s} \oplus J\om$. 
We set $\xi = \pi^{-1}\sqrt{s}\om(\sqrt{s}\om)$. 
Then $\xi$ gives an isomorphism as in \REF{is3} and $\xi\xi^{*} = -s^{2}$. 

If $\pe \nmid 2$, 
by \REF{lw41}, 
$\{\sqrt{s},\, \om,\, \sqrt{s}\om\}$ is an orthogonal $\g$-basis of $L \cap W$ with respect to $\psi$. 
Then straightforward calculations show that $(L \cap W)\xi = \pi(\alwd \cap \awc)$. 

If $\pe \mid 2$, 
we can put $s = (1 + \pi^{2\ka}a)b^{-2}$ with $a,\, b \in \gx$. 
Let $\alw = \gf\zeta$ with $\zeta$ given by \REF{alw41}. 
The matrix $\n_{0}$ representing $\n$ with respect to $\zeta$ is given by 
\begin{gather*}
\n_{0} = 
\begin{bmatrix}
   		1 & 0 & 0 & \pi^{1-\ka} \\ 
   		0 & \pi s & \pi^{1-\ka}s & 0 \\ 
  		0 & \pi^{1-\ka}s & -\pi sa & 0 \\
   		\pi^{1-\ka} & 0 & 0 & -\pi^{2}a \\ 
\end{bmatrix}. 
\end{gather*}
Then together with \REF{lw41}, 
$(L \cap W)\xi$ can be expressed as 
\begin{gather*}
(L \cap W)\xi = \gt 
\begin{bmatrix}
   		0 & 0 & 0 & -2\pi^{1-\ka}s^{2}b \\ 
   		0 & 0 & 2\pi^{1-\ka}s^{2}b & 0 \\ 
   		0 & -2\pi^{1-\ka}s^{2}b & 0 & 0 \\
\end{bmatrix}(2\n_{0})^{-1}\zeta. 
\end{gather*}
From this we have $(L \cap W)\xi = \pi(\alwd \cap \awc)$. 
This settles the case $t = 4$. 
\\

Let us complete the proof of \THM{th2}. 
In all the cases needed in the proof we have seen that 
\begin{gather}
(L \cap W)\xi = \pi^{\mu}(\alwd \cap \awc) \label{src}
\end{gather}
with a specified isomorphism $\xi$ in \REF{is3} and $\mu = \la - m$ or $\la + 1 - m$. 
In each case observe that 
$\mu = \la + 1 - m$ if and only if $D_{Q(\psi)} = \pe$ or if $D_{Q(\psi)} = \g$ and $\mathfrak{a} = \pe$ 
with $\mathfrak{a}$ in the statement of \THM{th2}. 
Also, 
$\aw$ is a division algebra if and only if $t = 4$ or if $t = 2$ and $q \not\in \p[Kg]$. 
In view of these, 
\REF{src} can be given by 
$(L \cap W)\xi = \mathfrak{c}\pe^{\la-m}(\alwd \cap \awc)$ 
with $\mathfrak{c}$ in the statement. 

Let $\zeta$ be an arbitrary isomorphism in \REF{is3} such that $\zeta\zeta^{*} \in \gx \cup \pi\gx$. 
Notice that $F(\sqrt{-\zeta\zeta^{*}})$ is the discriminant field of $\psi$. 
Then $\xi^{-1}\zeta \in F^{\times}$ as mentioned in {\S}1.2. 
Because $\xi\xi^{*} \in \gx \cup \pi\gx$ for every case, 
$\xi^{-1}\zeta$ belongs to $\gx$. 
Hence we have 
$(L \cap W)\zeta = \mathfrak{c}\pe^{\la-m}(\alwd \cap \awc)\xi^{-1}\zeta 
	= \mathfrak{c}[M/L \cap W](\alwd \cap \awc)$, 
where $M$ is a maximal lattice in $W$ with respect to $\psi$. 
This completes the proof of \THM{th2}.  

\section{Theorem on $\lw$ in the global case}

\begin{thm} \label{th1}
Let $(V,\, \p)$ be a quaternary quadratic space over a number field $F$. 
For $h \in V$ such that $\p[h] = q \ne 0$, 
put $W = (Fh)^{\perp}$ and let $\psi$ be the restriction of $\p$ to $W$. 
Fix an isomorphism $\xi$ of $(W,\, \psi)$ onto $(A^{+}(W)^{\circ},\, (\xi\xi^{*})^{-1}\nu^{\circ})$ 
as stated in \REF{is3}. 
Put $\xi\xi^{*}\g = \mathfrak{a}\mathfrak{r}^{2}$ 
with a squarefree integral ideal $\mathfrak{a}$ and a $\g$-ideal $\mathfrak{r}$ of $F$. 
Then, 
for a $\g$-maximal lattice $L$ in $V$ with respect to $\p$, 
$L \cap W$ can be given by 
\begin{gather}
(L \cap W)\xi = \mathfrak{r}\mathfrak{c}D_{Q(\psi)}[M/L \cap W]\cdot
	\{A^{+}(L \cap W)\widetilde{\,} \cap A^{+}(W)^{\circ}\}. \label{rec}
\end{gather}
Here $\mathfrak{c}$ is the product of the prime ideals of $\mathfrak{a}$ 
that are prime to the discriminant $D_{Q(\psi)}$ of $A^{+}(W)$ 
and $M$ is a $\g$-maximal lattice in $W$ with respect to $\psi$. 
\end{thm}
\begin{proof}
Let $K = F(\sqd)$ be the discriminant field of $\p$. 
We first observe that $\dl q\g = \mathfrak{a}(\mr a)^{2}$ with some $a \in F^{\times}$ 
because $\dl q \in \xi\xi^{*}F^{\times 2}$. 
\THM{th2} combined with the localization at $v \in \mathbf{h}$ proves that 
$(L \cap W)_{v}\zeta_{v} = (\mathfrak{c}D)_{v}[M/L \cap W]_{v}(\alwd_{v} \cap \awc_{v})$, 
where $D$ is the discriminant of $\aw$ and 
$\zeta_{v}$ is an isomorphism in \REF{is3} such that $\zeta_{v}\zeta_{v}^{*}\g_{v} = \mathfrak{a}_{v}$. 
Then $\xi = c\zeta_{v}$ with some $c \in F_{v}^{\times}$. 
Since $c^{2}\mathfrak{a}_{v} = \xi\xi^{*}\g_{v} = \mathfrak{a}_{v}\mr_{v}^{2}$, 
we have $\zeta_{v}^{-1}\xi\g_{v} = \mr_{v}$. 
Thus, through the $\xi$, 
$(L \cap W)_{v}$ can be given by 
\begin{eqnarray*}
(L \cap W)_{v}\xi &=& (\mathfrak{c}D)_{v}[M/L \cap W]_{v}\cdot\{\alwd_{v} \cap \awc_{v}\}\zeta_{v}^{-1}\xi \\
&=& \mr_{v}(\mathfrak{c}D)_{v}[M/L \cap W]_{v}\cdot\{\alwd_{v} \cap \awc_{v}\} 
\end{eqnarray*}
for every $v \in \mathbf{h}$. 
This determines \REF{rec}. 
Our theorem is thereby proved. 
\end{proof}

\begin{cor} \label{co1}
Let the notation be the same as in \THM{th1}. 
Put $\mathfrak{o} = A^{+}(L \cap W)$. 
Then the following assertions hold: 
\begin{enumerate}
\item The $SO^{\psi}$-genus of $L \cap W$ is determined by the genus of $\mathfrak{o}$ which is defined by the set 
	$\{\alpha^{-1}\mathfrak{o}\alpha \mid \alpha \in A^{+}(W)^{\times}_{\mathbf{A}}\}$. 
\item $C(L \cap W) = \tau(T(\mathfrak{o}))$ and $\Gamma(L \cap W) = \tau(\Gamma^{*}(\mathfrak{o}))$. 
\item The map $N \mapsto A^{+}(N)$ gives a bijection of the set of $SO^{\psi}(W)$-classes 
	in the genus of $\lw$ onto the set of isomorphism classes in the genus of $\mathfrak{o}$. 
\item The class number of the genus of $L \cap W$ with respect to $SO^{\psi}(W)$ 
	equals the type number of $\mathfrak{o}$ whose discriminant is $q[\widetilde{L}/L](2\p(h,\, L))^{-2}$. 
\end{enumerate}
\end{cor}
\begin{proof}
Assertion (1) is verified from \REF{rec}. 
To prove the first assertion of (2), 
we have only to verify that $C(L \cap W) = \tau(T(\mo))$ in the local situation. 
Suppose that $(L \cap W)\tau(\alpha) = L \cap W$ for $\alpha \in G^{+}(W) = \aw^{\times}$. 
Then $\alpha\mathfrak{o}\alpha^{-1} = A^{+}(\alpha(L \cap W)\alpha^{-1}) = \mathfrak{o}$, 
and hence $C(L \cap W) \subset \tau(T(\mo))$. 
Conversely, 
if $\alpha \in T(\mo)$, 
then 
\begin{gather*}
2\n(\alpha^{-1}x\alpha,\ \mo) = 2\n(x,\ \alpha\mo\alpha^{-1}) = 2\n(x,\ \mo) \subset \g 
\end{gather*}
for every $x \in \widetilde{\mo}$. 
Hence $\alpha^{-1}\widetilde{\mo}\alpha \subset \widetilde{\mo}$. 
A similar argument proves $\alpha^{-1}\widetilde{\mo}\alpha = \widetilde{\mo}$. 
Thus by \REF{rec}, 
$(L \cap W)\tau(\alpha) = \mathfrak{r}\mathfrak{c}D_{Q(\psi)}[M/L \cap W]
	(\alpha^{-1}\widetilde{\mo}\alpha \cap \awc)\xi^{-1} = L \cap W$. 
This shows $\tau(T(\mo)) \subset C(L \cap W)$. 
Thus we obtain the desired fact. 
In a similar way we have the second assertion of (2). 

The homomorphism $\tau$ in \REF{tau} gives a bijection of 
$T(\mathfrak{o})\setminus \aw^{\times}_{\mathbf{A}}/$ $\aw^{\times}$ 
onto $C(L \cap W)\setminus SO^{\psi}(W)_{\mathbf{A}}/SO^{\psi}(W)$, 
which proves (3). 
This assertion can also be verified by employing \REF{rec} in a straightforward way. 
Notice that every lattice $N$ belonging to the genus of $\lw$ is integral with respect to $\psi$. 
Assertion (4) is a consequence from (3). 
\end{proof}
This corollary proves \cite[Theorem 2.2]{qeo}. 
Also \LEM{lalw} together with \REF{od} proves \cite[Lemma 2.1]{qeo} by localization. 
Applying that result and \THM{th1} to the case of $F=\mathbf{Q}$, 
we have \cite[Lemma 2.1]{pr}.

Manabu\ Murata 

College of Science and Engineering 

Ritsumeikan University 

Kusatsu,\ Shiga 525-8577 

Japan 

murata31@pl.ritsumei.ac.jp 

\end{document}